\documentclass[a4paper,10pt]{amsart}
\usepackage[UKenglish]{babel}
\usepackage{amsmath,amssymb,amsthm}
\usepackage[pdftex]{color}
\usepackage[bookmarks=true,hyperindex,pdftex,colorlinks, citecolor=MidnightBlue,linkcolor=MidnightBlue, urlcolor=MidnightBlue]{hyperref}
\usepackage[dvipsnames]{xcolor}
\usepackage[shortlabels]{enumitem}
\usepackage{graphicx}
\usepackage{tikz-cd} %diagramas
\usepackage{tikz}
\usetikzlibrary{arrows}
\usepackage{longfbox}

\usetikzlibrary{decorations.markings}
\tikzset{neg/.style={
            decoration={markings,
            mark= at position 0.5 with {
                  \node[transform shape] (tempnode) {$\setminus$};
                  }
              },
              postaction={decorate}
}}

\newtheorem{theorem}{Theorem}[section]
\newtheorem{lemma}[theorem]{Lemma}

\newtheorem{proposition}[theorem]{Proposition}
\newtheorem{corollary}[theorem]{Corollary}

\theoremstyle{definition}
\newtheorem{definition}[theorem]{Definition}

\newtheorem{example}[theorem]{Example}

\theoremstyle{remark}
\newtheorem{remark}[theorem]{Remark}
\numberwithin{equation}{section}

\newcommand{\real}{\mathbb{R}}
\newcommand{\R}{\mathbb{R}}
\newcommand{\N}{\mathbb{N}}

\newcommand{\pten}{\ensuremath{\widehat{\otimes}_\pi}}

  \DeclareMathOperator{\diam}{diam\,}
  \DeclareMathOperator{\co}{co}
  \newcommand{\cconv}{\overline{\co}}

\renewcommand{\geq}{\geqslant}
\renewcommand{\leq}{\leqslant}

\newcommand{\norm}[1]{\left\Vert#1\right\Vert}

\newcommand{\lip}{{\mathrm{lip}}_0}
\newcommand{\length}{{\mathrm{length}}}
\newcommand{\Lip}{{\mathrm{Lip}}_0}

\newcommand{\SA}{\operatorname{SNA}}
\newcommand{\NA}{\operatorname{NA}}

\newcommand{\ext}[1]{\operatorname{ext}\left(#1\right)}
\newcommand{\strexp}[1]{\operatorname{str-exp}\left(#1\right)}
\newcommand{\preext}[1]{\operatorname{pre-ext}\left(#1\right)}
\newcommand{\dent}[1]{\operatorname{dent}\left(#1\right)}
\newcommand{\Mol}[1]{\operatorname{Mol}\left(#1\right)}

\newcommand{\supp}{\operatorname{supp}}
\newcommand\restr[2]{\ensuremath{\left.#1\right|_{#2}}}

\title{On strongly norm attaining Lipschitz maps}

\author[Cascales]{Bernardo Cascales\texorpdfstring{$^\dag$}{}}
\thanks{$\dag$ Sadly, Bernardo Cascales passed away in April, 2018. As this work was initiated with him, the rest of the authors decided to finish the research and to submit the paper with his name as coauthor. This is our tribute to our dear friend and master}
\address[Cascales]{Departamento de Matem\'aticas, Universidad de Murcia, 30100 Espinardo, Murcia, Spain}

\author[Chiclana]{Rafael Chiclana}

\author[Garc\'ia-Lirola]{Luis C.\ Garc\'ia-Lirola}
\address[Garc\'ia-Lirola]{Department of Mathematical Sciences, Kent State University, Kent OH 44242, USA}
\email{lucagali@gmail.com}

\author[Mart\'in]{Miguel Mart\'in}

\author[Rueda Zoca]{Abraham Rueda Zoca}
\address[Chiclana, Mart\'in, Rueda-Zoca]{Universidad de Granada, Facultad de Ciencias.
Departamento de An\'{a}lisis Matem\'{a}tico, 18071-Granada
(Spain)}
\email{rafachiclanavega@gmail.com; mmartins@ugr.es; abrahamrueda@ugr.es}

\thanks{The research of B.\ Cascales and L.\ Garc\'ia-Lirola was supported by the grants Spanish MTM2017-83262-C2-2-P and Fundaci\'on S\'eneca CARM 19368/PI/14. L. Garc\'ia-Lirola was also supported by a postdoctoral grant in the framework of \emph{Programa Regional de Talento Investigador y su Empleabilidad} from \emph{Fundaci\'on S\'eneca - Agencia de Ciencia y Tecnolog\'ia de la Regi\'on de Murcia}. The research of R.\ Chiclana and M.\ Mart\'{\i}n was supported by Spanish MINECO/FEDER grant MTM2015-65020-P and by Junta de Andaluc\'ia Grant FQM-0185. The research of A.\ Rueda Zoca was supported by MECD FPU16/00015, by Spanish MINECO/FEDER grant MTM2015-65020-P, and by Junta de Andaluc\'ia Grant FQM-0185.}



\date{July 9th, 2018. Revised: December 29th, 2018}

\let\oldtocsection=\tocsection

\let\oldtocsubsection=\tocsubsection

\renewcommand{\tocsection}[2]{\hspace{0em}\oldtocsection{#1}{#2}}
\renewcommand{\tocsubsection}[2]{\hspace{0,5em}\oldtocsubsection{#1}{#2}}

\begin{document}

\begin{abstract}
We study the set $\SA(M,Y)$ of those Lipschitz maps from a (complete pointed) metric space $M$ to a Banach space $Y$ which (strongly) attain their Lipschitz norm (i.e.\ the supremum defining the Lipschitz norm is a maximum). Extending previous results, we prove that this set is not norm dense when $M$ is a length space (or local) or when $M$ is a closed subset of $\R$ with positive Lebesgue measure, providing new examples which have very different topological properties than the previously known ones. On the other hand, we study the linear properties which are sufficient to get Lindenstrauss property A for the Lipschitz-free space $\mathcal{F}(M)$ over $M$, and show that all of them actually provide the norm density of $\SA(M,Y)$ in the space of all Lipschitz maps from $M$ to any Banach space $Y$. Next, we prove that $\SA(M,\R)$ is weakly sequentially dense in the space of all Lipschitz functions for all metric spaces $M$. Finally, we show that the norm of the bidual space of $\mathcal{F}(M)$ is octahedral provided the metric space $M$ is discrete but not uniformly discrete or $M'$ is infinite.
\end{abstract}

\maketitle

\thispagestyle{plain}

\begin{center}
\begin{minipage}[c]{0,9\textwidth}\small
\tableofcontents
\end{minipage}
\end{center}

\section{Introduction}

Our aim in this paper is to discuss when the set of those Lipschitz maps which strongly attain their norm is dense in the space of Lipschitz maps. Let us give the necessary definitions. A \emph{pointed metric space} is just a metric space $M$ in which we distinguish an element, called $0$. All along the paper, the metric spaces will be complete and the Banach spaces will be over the real scalars. Given a pointed metric space $M$ and a Banach space $Y$, we write $\Lip(M,Y)$ to denote the Banach space of all Lipschitz maps $F:M\longrightarrow Y$ which vanish at $0$, endowed with the Lipschitz norm defined by
\begin{equation}\label{eq:def-norma-Lipschitz}
 \| F \|_L := \sup\left\{\frac{\|F(x)-F(y)\|}{d(x,y)} \colon x,y\in M,\, x \neq y \right\}.
\end{equation}
Let us comment that the election of the distinguished element is not important, as the resulting spaces of Lipschitz maps are isometrically isomorphic. Following \cite{kms} and \cite{Godefroy-survey-2015}, we say that $F\in \Lip(M,Y)$ \emph{attains its norm in the strong sense} or that \emph{strongly attains its norm}, whenever the supremum in \eqref{eq:def-norma-Lipschitz} is actually a maximum, that is, whenever there are $x,y\in M$, $x\neq y$, such that
$$
\frac{\|F(x)-F(y)\|}{d(x,y)}=\|F\|_L.
$$
The subset of all Lipschitz maps in $\Lip(M,Y)$ which attain their norm in the strong sense is denoted by $\SA(M,Y)$.

As far as we know, the study of norm attaining Lipschitz maps was initiated independently in \cite{Godefroy-PAFA-2016} and \cite{kms}. Both papers deal with notions of norm attainment which are different from the strong one, and they are focused on Lipschitz functionals (\cite{kms}) or on vector-valued Lipschitz maps (\cite{Godefroy-PAFA-2016}) defined on Banach spaces. The paper \cite{Godefroy-PAFA-2016} contains some negative results on the density of the set of Lipschitz maps which attain their norm in a very weak way. The paper \cite{kms} contains positive results on the density of the set of Lipschitz functionals which attain their norm ``directionally'', a notion which is weaker than the strong norm attainment. It also contains negative results: when $M$ is a Banach space, $\SA(M,\R)$ is not dense in $\Lip(M,\R)$, and it is also the case when $M=[0,1]$ or, more generally, when $M$ is metrically convex (or geodesic) (see Section \ref{sect:negative} for the definition). Our first aim in this paper will be to extend these negative results to more general metric spaces as length spaces and subsets of $[0,1]$ with positive Lebesgue measure, see the details in Section  \ref{sect:negative}. As a consequence of our results, we will obtain examples of metric spaces $M$ where $\SA(M,\R)$ is not dense in $\Lip(M,\R)$ and, in contrast with the previously known results, no connectedness assumption is needed on $M$ (e.g.\ we can consider $M$ to be a ``fat'' Cantor set).

On the other hand, the paper \cite{Godefroy-survey-2015} contains the first positive result on the density of strongly norm attaining Lipschitz functionals (and also some results for Lipschitz maps): this is the case when the little Lipschitz space over $M$ strongly separates $M$ (as it is the case of $M$ being compact and countable \cite{dalet}, when $M$ is the middle third Cantor set, or when $M$ is a compact H\"older space \cite[Proposition~3.2.2]{wea5}). A slight generalisation can be found in \cite[Section 4]{gprstu}. These results have been recently extended in \cite[Proposition 7.4]{lppr}, but we need a little more background in order to enunciate the result. Let $M$ be a pointed metric space. We denote by $\delta$ the canonical isometric embedding of $M$ into $\Lip(M,\R)^*$, which is given by $\langle f, \delta(x) \rangle =f(x)$ for $x \in M$ and $f \in \Lip(M,\R)$. We denote by $\mathcal{F}(M)$ the norm-closed linear span of $\delta(M)$ in the dual space $\Lip(M,\R)^*$, which is usually called the \textit{Lipschitz-free space over $M$}, see the papers \cite{Godefroy-survey-2015} and \cite{gk}, and the book \cite{wea5} (where it receives the name of Arens-Eells space) for background on this. It is well known that $\mathcal{F}(M)$ is an isometric predual of the space $\Lip(M,\R)$ \cite[pp. 91]{Godefroy-survey-2015}, indeed it is the unique isometric predual when $M$ is bounded or a geodesic space \cite{Weaver-2017}. Now, \cite[Proposition 7.4]{lppr} states that if $\mathcal{F}(M)$ has the Radon-Nikod\'{y}m property (RNP in short), then $\SA(M,Y)$ is dense in $\Lip(M,Y)$ for every Banach space $Y$, extending by far the results of \cite{Godefroy-survey-2015}. At the beginning of Section \ref{sec:property_A} we will give a short exposition of why this result holds. Examples of metric spaces for which $\mathcal{F}(M)$ has the RNP are exhibited in Example \ref{ejernp}.

There is a connection between the study of the density of norm attaining Lipschitz maps and the study of norm attaining linear operators, a research line which goes back to Lindenstrauss' seminal paper \cite{lindens} from 1963. Let us give a piece of notation for Banach spaces. Given a Banach space $X$, we will denote by $B_X$ and $S_X$ the closed unit ball and the unit sphere of $X$, respectively. We will also denote by $X^*$ the topological dual of $X$. If $Y$ is another Banach space, we write $\mathcal{L}(X,Y)$ to denote the Banach space of all bounded linear operators from $X$ to $Y$, endowed with the operator norm. We say that $T\in \mathcal{L}(X,Y)$ \emph{attains its norm}, and write $T\in \NA(X,Y)$, if there is $x\in X$ with $\|x\|=1$ such that $\|Tx\|=\|T\|$. The study of the density of norm attaining linear operators has its root in the classical Bishop-Phelps theorem which states that $\NA(X,\R)$ is dense in $X^*=\mathcal{L}(X,\R)$ for every Banach space $X$. J.~Lindenstrauss extended such study to general linear operators, showed that this is not always possible, and also gave positive results. If we say that a Banach space $X$ has \emph{(Lindenstrauss) property A} when $\overline{\NA(X,Y)}=\mathcal{L}(X,Y)$ for every Banach space $Y$, it is shown in \cite{lindens} that reflexive spaces have property A. This result was extended by J.~Bourgain \cite{bourgain1977} showing that Banach spaces $X$ with the RNP also have Lindenstrauss property A, and he also provided a somehow reciprocal result. We refer the interested reader to the survey paper \cite{AcostaRACSAM2006} for a detailed account on norm attaining linear operators.

Coming back to Lipschitz maps, let us recall that when $M$ is a pointed metric space and $Y$ is a Banach space, it is well known that every Lipschitz map $f \colon M \longrightarrow Y$ can be isometrically identified with the continuous linear map $\widehat{f} \colon \mathcal{F}(M) \longrightarrow Y$ defined by $\widehat{f}(\delta_p)=f(p)$ for every $p \in M$. This mapping completely identifies the spaces $\Lip(M,Y)$ and $\mathcal{L}(\mathcal{F}(M),Y)$. Bearing this fact in mind, the set $\SA(M,Y)$ is identified with the set of those elements of $\mathcal{L}(\mathcal{F}(M),Y)$ which attain their operator norm at elements of the form $\frac{\delta(x)-\delta(y)}{d(x,y)}$ for some $x,y\in M$, $x\neq y$. It then follows that when $\SA(M,Y)$ is dense in $\Lip(M,Y)$, in particular, $\NA(\mathcal{F}(M),Y)$ has to be dense in $\mathcal{L}(\mathcal{F}(M),Y)$. The converse result is not true as, for instance, $\NA(\mathcal{F}(M),\R)$ is always dense by the Bishop-Phelps theorem but, as we have already mentioned, there are many metric spaces $M$ such that $\SA(M,\R)$ is not dense in $\Lip(M,\R)$. Of course, if $\SA(M,Y)$ is dense in $\Lip(M,Y)$ for every Banach space $Y$, then $\mathcal{F}(M)$ has Lindenstrauss property A. We do not know whether the converse result is true, but it is now not very surprising the appearance of the RNP of $\mathcal{F}(M)$ as a sufficient condition for the density of $\SA(M,Y)$ in $\Lip(M,Y)$ for every $Y$ \cite[Proposition 7.4]{lppr}. Actually, as far as we know, the RNP of $\mathcal{F}(M)$ could be a necessary condition for the density of $\SA(M,Y)$ in $\Lip(M,Y)$ for every space $Y$. On the other hand, there are several geometric properties of a Banach space $X$ which imply Lindenstrauss property A, being the most common, apart from having $X$ the RNP, the properties $\alpha$ and quasi-$\alpha$ and the existence of a uniformly strongly exposed set of $B_X$ whose closed convex hull is the whole $B_X$. In Section \ref{sec:property_A}, we analyse these properties for Lipschitz-free spaces and show that each of them actually forces $\SA(M,Y)$ to be dense in $\Lip(M,Y)$ for every Banach space $Y$. We also provide characterisations of these properties for $\mathcal{F}(M)$ in terms of the metric space $M$ and study the relationship between them. To this end, one of the main tools will be the recent characterisations of strongly exposed points and denting points of the unit ball of Lipschitz-free spaces appearing in \cite{gpr} and \cite{lppr}, respectively, which we will include at Subsection \ref{subsec:geometry-free-spaces}.

The previous results make clear that the density of $\SA(M,\R)$ in $\Lip(M,\R)$ is a strong requirement and there are not too many metric spaces having this property. A completely different situation holds when we deal with weak density: we show in Section \ref{sectidensidadebil} that $\SA(M,\R)$ is weakly sequentially dense in $\Lip(M,\R)$ for every pointed metric space $M$, extending \cite[Theorem 2.6]{kms}, where the result was proved when $M$ is a Banach space or, more generally, when $M$ is a length space.

The main tool to get the above result is an extension of a lemma from \cite{kms} which provides an easy criterium to get weak convergence of a sequence of Lipschitz maps which we include in Subsection \ref{subsec:geometry-free-spaces}. Such a result produces a by-product of our study: that the norm of the bidual of $\mathcal{F}(M)$ is octahedral when $M'$ is infinite or $M$ is discrete but not uniformly discrete. Recall that the norm of a Banach space $X$ is said to be \emph{octahedral} if, given a finite-dimensional subspace $Y$ of $X$ and $\varepsilon>0$, we can find $x\in S_X$ such that the inequality
$$\Vert y+\lambda x\Vert\geq (1-\varepsilon)(\Vert y\Vert +\vert\lambda\vert)$$
holds for every $y\in Y$ and $\lambda\in\mathbb R$. From an isomorphic point of view, it was proved in \cite{godefroyocta} that a Banach space $X$ can be equivalently renormed to have an octahedral norm if, and only if, the space $X$ contains an isomorphic copy of $\ell_1$ and it was left as an open problem whether any Banach space containing an isomorphic copy of $\ell_1$ can be equivalently renormed so that the bidual norm is octahedral. From an isometric point of view, it is
proved in \cite[Theorem~1 and Proposition 3]{deville} that if a Banach space $X$ has an octahedral norm, then every convex combination of weak-star slices of $B_{X^*}$ has diameter two, and the reciprocal result has been recently proved in \cite{blroctajfa}. By using this characterisation, it was proved in \cite[Theorem 2.4]{blr} that if $M$ is not uniformly discrete and bounded, then $\mathcal F(M)$ has an octahedral norm. This result was pushed further in \cite{pr}, where the authors characterised all the Lipschitz-free Banach spaces $\mathcal F(M)$ whose norm is octahedral in terms of a geometric property of the underlying metric space $M$. Observe that the norm of $X^{**}$ is octahedral if and only if every convex combination of weak slices of $B_{X^*}$ has diameter two \cite[Corollary 2.2]{blroctajfa}. Thus, easy examples show that the norm of a Banach space $X$ can be octahedral without its bidual norm being octahedral (e.g.\ $X=\mathcal C([0,1])$ does the work). It is then a natural question to check when the bidual norm of $\mathcal F(M)$ can be octahedral. Particular examples of metric spaces $M$ satisfying that the norm of $\mathcal{F}(M)^{**}$ is octahedral are known (for instance when $M$ is a subset of an $\R$-tree as a consequence of \cite[Theorem 4.2]{Godard} and \cite[Proposition 3.4]{Yagoub}). However, it is not known whether there exists a metric space $M$ such that the norm of $\mathcal{F}(M)$ is octahedral but that of $\mathcal{F}(M)^{**}$ is not. As we have already announced, we will prove that the norm of the bidual space of $\mathcal{F}(M)$ is octahedral when $M'$ is infinite or $M$ is a discrete but not uniformly discrete metric space. Besides, as a consequence of the techniques involved in the proof, we obtain a partial positive answer to \cite[Question 3.1]{blr}.

Even though all the main results of the paper have been presented so far, we would like to include an outline of the paper. We finish this introduction with a subsection including the needed notation and terminology on metric spaces and some new and previously known results on the geometry of Lipschitz-free spaces which will be relevant in our discussion. In Section \ref{sect:negative} we extend the negative examples of \cite{kms} to more general ones: we prove that $\SA(M,\R)$ is not dense in $\Lip(M,\R)$ whenever $M$ is a length metric space and when $M$ is a subset of an $\mathbb R$-tree with positive measure, in particular, when $M$ is a subset of $\R$ with positive Lebesgue measure. We devote Section \ref{sec:property_A} to discuss some sufficient conditions for Lindenstrauss property A in the setting of Lipschitz-free spaces, showing that all of them actually imply the density of strongly norm attaining Lipschitz maps; we also give metric characterisations of some of them and discuss the relations between them. The main result of Section \ref{sectidensidadebil} is that $\SA(M,\R)$ is weakly sequentially dense in $\Lip(M,\R)$ for every pointed metric space $M$. Finally, we show in Section \ref{sectiocta} that the norm of $\mathcal{F}(M)^{**}$ is octahedral when $M'$ is infinite or $M$ is discrete but not uniformly discrete.

\subsection[Geometry of Lipschitz-free spaces]{New and old results on the geometry of Lipschitz-free spaces}\label{subsec:geometry-free-spaces}
Let $X$ be a Banach space. A \emph{slice} of the unit ball $B_X$ is a non-empty intersection of an open half-space with $B_X$; all slices can be written in the form \[
S(B_X,f,\beta):=\{x\in B_X \colon f(x)>1-\beta\}
\]
where $f \in S_{X^*}$, $\beta>0$. The notations $\ext{B_X}$, $\preext{B_X}$, $\strexp{B_X}$ stand for the set of extreme points, preserved extreme points (i.e.\ extreme points which remain extreme in the bidual ball), and strongly exposed points of $B_X$, respectively. A point $x \in B_X$ is said to be a \emph{denting point} of $B_X$ if there exist slices of $B_X$ containing $x$ of arbitrarily small diameter. We will denote by $\dent{B_X}$ the set of denting points of $B_X$. We always have that
$$
\strexp{B_X}\subset \dent{B_X} \subset \preext{B_X} \subset \ext{B_X}.
$$

Given a metric space $M$, $B(x,r)$ denotes the closed ball in $M$ centered at $x \in M$ with radius $r$. Given $x, y \in M$, we write $[x,y]$ to denote the \emph{metric segment} between $x$ and $y$, that is,
\[
[x,y] := \{z \in M \colon d(x,z)+d(z,y)=d(x,y)\}.
\]
By a \emph{molecule} we mean an element of $\mathcal{F}(M)$ of the form
\[
m_{x,y}:=\frac{\delta(x)-\delta(y)}{d(x,y)}
\]
for $x, y \in M$, $x \neq y$. We write $\Mol{M}$ to denote the set of all molecules of $M$. Note that, since $\Mol{M}$ is balanced and norming for $\Lip(M,\R)$, a straightforward application of Hahn-Banach theorem implies that
$$
\overline{\co}(\Mol{M})=B_{\mathcal F(M)}
$$
(the notation $\overline{\co}(A)$ denotes the closed convex hull of a set $A$).

The following proposition summarises some known results about extremality in Lipschitz-free spaces that we may find in \cite[Corollary 2.5.4]{wea5}, \cite[Theorem 2.4]{lppr}, \cite[Theorem 5.4]{gpr}, and \cite[Proposition 2.9]{lppr}. We need some notation: given $x,y,z \in M$, the \emph{Gromov product of $x$ and $y$ at $z$} is defined as
 \[
  (x,y)_z:=\frac12\bigl(d(x,z)+d(y,z)-d(x,y)\bigr)\geq 0,
 \]
see e.g.~\cite{bh}. It corresponds to the distance of $z$ to the unique closest point $b$ on the unique geodesic between $x$ and $y$ in any $\mathbb R$-tree into which $\{x,y,z\}$ can be isometrically embedded (such a tree, tripod really, always exists). Notice that $(x,z)_y+(y,z)_x=d(x,y)$ and that $(x,y)_z\leq d(x,z)$, facts which we will use without further comment.

\begin{proposition}\label{prop:extremalidad}
Let $M$ be a complete metric space. Then:
\begin{enumerate}[(a)]
\item Every preserved extreme point of $B_{\mathcal F(M)}$ is both a molecule and a denting point of $B_{\mathcal F(M)}$, so $$\preext{B_{\mathcal F(M)}}=\dent{B_{\mathcal F(M)}}\subseteq \Mol{M}.$$
\item Given $x,y\in M$ with $x\neq y$, the following assertions are equivalent:
\begin{enumerate}[(i)]
\item $m_{x,y}$ is a strongly exposed point of $B_{\mathcal F(M)}$.
\item There exists $\varepsilon_0>0$ such that the inequality
$$
(x,y)_z\geq \varepsilon_0 \min\{d(x,z),d(y,z)\}
$$
holds for every $z\in M\setminus \{x,y\}$ (in other words, the pair $(x,y)$ fails property $(Z)$ of \cite{ikw}).
\end{enumerate}
\item  $\Mol{M}$ is closed in $\mathcal F(M)$.
\end{enumerate}
\end{proposition}

According to \cite[p.\ 51]{wea5}, a metric space $M$ is said to be \emph{concave} if every molecule $m_{x,y}\in \Mol{M}$ is a preserved extreme point of $B_{\mathcal F(M)}$. Thanks to the characterisation of the preserved extreme points given in \cite{ag}, a metric space $M$ is concave if, and only if, for every $x,y\in M$ and every  $\varepsilon>0$, there exists $\delta>0$ such that the inequality
$$
d(x,z)+d(y,z)>d(x,y)+\delta
$$
holds for every $z\in M$ such that $\min\{d(x,z),d(y,z)\}\geq \varepsilon$. It is known that every H\"older metric space is a concave metric space \cite[p.~51]{wea5}.
Recall that a \emph{H\"older metric space} is $(M,d^\theta)$ where $(M,d)$ is a metric space and $0<\theta<1$.
We refer the reader to \cite{Kalton04} and \cite{wea5} for background on H\"older metric spaces and the structure
of their Lipschitz-free spaces. Moreover, \cite[Corollary 4.4]{ag} yields that a compact metric space $M$ is concave if and only if $d(x,z)+d(z,y)>d(x,y)$ for every $x,y,z$ distinct points in $M$, that is, if $[x,y]=\{x,y\}$ for every $x,y\in M$.

In general, examples of Banach spaces with a rich extremal structure are those with the RNP. Because of this reason and its relation with strongly norm attainment (see Theorem \ref{teornpdensidad}), let us exhibit some known examples of Lipschitz-free spaces with the RNP.

\begin{example}\label{ejernp}
The space $\mathcal F(M)$ has the RNP in the following cases:
\begin{enumerate}[(a)]
\item $M$ is uniformly discrete (i.e.\ $\inf_{x\neq y}d(x,y)>0$); \cite[Proposition 4.4]{Kalton04}.
\item $M$ is a countable compact metric space (since, in this case, $\mathcal F(M)$ is a separable dual Banach space); \cite[Theorem~2.1]{dalet}.
\item $M$ is a compact H\"older metric space (since, in this case, $\mathcal F(M)$ is a separable dual Banach space); \cite[Corollary~3.3.4]{wea5}.
\item $M$ is a closed subset of $\mathbb{R}$ with measure $0$ (since, in this case, $\mathcal F(M)$ is isometric to $\ell_1$);  \cite{Godard}.
\end{enumerate}
\end{example}

The following lemma, coming from \cite{LCthesis}, provides a useful estimate of the norm of differences of molecules. For completeness, we will include a proof of the result.

\begin{lemma}\label{lemma:ineqmolec} Let $M$ be a metric space and $x,y,u,v\in M$, with $x\neq y$ and $u\neq v$. Then
\[ \norm{m_{x,y}-m_{u,v}}\leq 2\frac{d(x,u)+d(y,v)}{\max\{d(x,y),d(u,v)\}}. \]
If, moreover, $\norm{m_{x,y}-m_{u,v}}<1$, then
\[ \frac{\max\{d(x,u), d(y,v)\}}{\min\{d(x,y), d(u,v)\}}\leq \norm{m_{x,y}-m_{u,v}}. \]
\end{lemma}

\begin{proof}
The first inequality follows from the well-known one
\[ \norm{\frac{z}{\norm{z}} - \frac{w}{\norm{w}} } \leq 2 \frac{\norm{z-w}}{\max\{\norm{z},\norm{w}\}}, \]
which holds for $z, w\neq 0$ in any Banach space, applied to $z=\delta(x)-\delta(y)$ and $w=\delta(u)-\delta(v)$.

To prove the second inequality, we assume that $\Vert m_{x,y}-m_{u,v}\Vert<1$ and take $r:=\min\{d(x,y),d(x,u)\}$. We define $f(t):=\max\{r-d(t,x),0\}$ for every $t\in M$ and $g:=f-f(0)$. It follows that $g\in \Lip(M,\R)$ with $\Vert g\Vert_L\leq 1$, so we get that
$$\Vert m_{x,y}-m_{u,v}\Vert\geq \vert \langle g,m_{x,y}-m_{u,v}\rangle\vert\geq \frac{r}{d(x,y)},$$
from where $r<d(x,y)$. This implies that $r=d(x,u)$ and
$$\frac{d(x,u)}{d(x,y)}\leq \Vert m_{x,y}-m_{u,v}\Vert.$$
Changing the roles of the pairs, we get the proof of the lemma.
\end{proof}

We also need the following result coming from \cite{kaufmannPreprint}, which is not included in the final version of that paper \cite{kaufmann}. Given a family  $\{X_\gamma\colon \gamma\in\Gamma\}$ of Banach spaces, we will denote by $\left[\bigoplus\nolimits_{\gamma\in\Gamma}X_\gamma\right]_{\ell_1}$ the $\ell_1$-sum of the family.

\begin{proposition}[Proposition 5.1 in \cite{kaufmannPreprint}]\label{prop:kaufmann} Suppose that $M = \bigcup_{\gamma\in \Gamma} M_\gamma$  is a metric space with metric $d$, and suppose that there exists $0 \in M$ satisfying
\begin{enumerate}
\item $M_\gamma\cap M_\eta = \{0\}$ if $\gamma\neq\eta$, and
\item there exists $C\geq 1$ such that $d(x,0)+d(y,0)\leq Cd(x,y)$ for all $\gamma\neq \eta$, $x\in M_\gamma$ and $y\in M_\eta$.
\end{enumerate}
Then, $\mathcal F(M)$ is isomorphic to $\left[\bigoplus_{\gamma\in \Gamma} \mathcal F(M_\gamma)\right]_{\ell_1}$. If $C=1$ such an isomorphism can be chosen to be isometric.
\end{proposition}

The previous result motivates the following definition: if $M$ is a metric space which can be written as $M = \bigcup_{\gamma\in \Gamma} M_\gamma$ satisfying (1) and (2) in the statement of the previous proposition for $C=1$, we say that $M$ is the \emph{$\ell_1$-sum} of the family  $\{M_\gamma\}_{\gamma\in \Gamma}$.

The next lemma provides a criterion to get weak convergence of sequences of Lipschitz functionals and maps, for which the weak topology does not have any easy description. It is inspired by \cite[Lemma~2.4]{kms}, improves \cite[Corollary 2.5]{kms} and will be the key to prove the main results of Sections \ref{sectidensidadebil} and \ref{sectiocta}.

\begin{lemma}\label{lemmaweaknull}
Let $M$ be a pointed metric space, let $Y$ be a Banach space, and let $\{f_n\}$ be a sequence of functions in the unit ball of $\Lip(M,Y)$. For each $n\in \mathbb N$, we write $U_n:=\{x\in M \colon f_n(x)\neq 0\}$ for the support of $f_n$. If $U_n\cap U_m = \emptyset$ for every $n\neq m$, then the sequence $\{f_n\}$ is weakly null.
\end{lemma}

\begin{proof}
We will show that for every finite collection of reals $\{a_j\}_{j=1}^n$, we have
\[  \norm{\sum_{j=1}^n a_j f_j}_L \leq 2\max_j |a_j|\]
and so $Te_n:= f_n$ defines a bounded linear operator from $c_0$ to $\Lip(M,Y)$.

To this end, denote $f= \sum_{j=1}^n a_j f_j$. Take $x, y\in M$ with $x\neq y$, and let us give an upper estimate for $\frac{\|f(x)-f(y)\|}{d(x,y)}$. Since the supports of the functions $\{f_n\}$ are pairwise disjoint, there are $j_1,j_2\in \{1,\ldots, n\}$ such that $\{x,y\}\cap U_j = \emptyset$ if $j\in\{1,\ldots,n\}\setminus\{j_1,j_2\}$. Therefore,
\begin{align*}
\frac{\|f(x)-f(y)\|}{d(x,y)} &= \frac{\|a_{j_1}(f_{j_1}(x)-f_{j_1}(y))+a_{j_2}(f_{j_2}(x)-f_{j_2}(y))\|}{d(x,y)} \\
&\leq |a_{j_1}|+|a_{j_2}|\leq 2\max_j |a_j|.
\end{align*}
This shows that the operator $T$ defined above is bounded. Thus, it is also weak-to-weak continuous and the conclusion follows.
\end{proof}

Finally, it is convenient to recall an important tool to construct Lipschitz functions: the classical McShane's extension theorem. It says that if $N \subseteq M$ and $f \colon N \longrightarrow \mathbb{R}$ is a Lipschitz function, then there is an extension to a Lipschitz function $F \colon M \longrightarrow \mathbb{R}$ with the same Lipschitz constant; see for example \cite[Theorem~1.5.6]{wea5}.

\section{New negative results}\label{sect:negative}

In this section we will exhibit new examples of metric spaces $M$ such that $\SA(M,\R)$ is not dense in $\Lip(M,\R)$. As we commented in the introduction, it was shown in \cite[Example 2.1]{kms} that $\SA([0,1],\R)$ is not dense in $\Lip([0,1],\R)$ and that this is extended in the same paper to all metrically convex pointed metric spaces \cite[Theorem 2.3]{kms}. Let us introduce some notation. A metric space $(M,d)$ is said to be a \textit{length space} if $d(x,y)$ is equal to the infimum of the length of the rectifiable curves joining $x$ and $y$ for every pair of points $x,y\in M$. In the case that such an infimum is actually a minimum, it is said that $M$ is \textit{geodesic} (or \emph{metrically convex}). It is clear that every geodesic space is a length space, but Example 2.4 in \cite{ikw} shows that the converse is not true. On the other hand, length spaces have been recently considered in \cite{gpr} where it is proved that a metric space $M$ is a length space if, and only if, $\Lip(M,\R)$ has the Daugavet property \cite[Theorem 3.5]{gpr}. Note by passing that for a complete metric space $M$, being a length space is also equivalent to the fact that every Lipschitz function on $M$ approximates its Lipschitz norm at points which are arbitrarily close (that is, $M$ is \emph{local}), see \cite[Proposition~3.4]{gpr}.

In this section we will consider two different generalisations of the fact from \cite{kms} that $\SA(M,\R)$ is not dense in $\Lip(M,\R)$ when $M$ is metrically convex. Our first aim is to replace metrically convex with being a length space in this result. To this end, we will need the following technical lemma, which is a generalisation of Lemma~2.2 in \cite{kms}.
	
	\begin{lemma}\label{lemlength}
Let $M$ be a pointed metric space, let $f \in \SA(M,\R)$ which attains its norm at a pair $(p,q)$ of different elements of $M$, let $\varepsilon>0$, and let $\alpha_\varepsilon$ be a rectifiable curve in $M$ joining $p$, $q$ such that
		\[ \length(\alpha_\varepsilon)\leq d(p,q) + \varepsilon.\]
Then, we have that
		\[ |f(z_1)-f(z_2)| \geq \lVert f \rVert_L (d(z_1,z_2)-\varepsilon) \quad \forall \, z_1, z_2 \in \alpha_\varepsilon.\]
	\end{lemma}

	\begin{proof}
		Fix $z_1$, $z_2 \in \alpha_\varepsilon$. By the definition of length of a curve, we have that
		\[ d(p,q)\leq d(p,z_1)+d(z_1,z_2)+d(z_2,q)\leq \length(\alpha_\varepsilon)\leq d(p,q)+\varepsilon.\]
		Consequently,
		\begin{align*}
		|f(z_1)-f(z_2)|&=|(f(p)-f(q))- ((f(p)-f(z_1))+(f(z_2)-f(q)))|\\
		&\geq |f(p)-f(q)|-|f(p)-f(z_1)|-|f(z_2)-f(q)| \\
        &\geq |f(p)-f(q)|-\lVert f \rVert_L d(p,z_1) - \lVert f \rVert_L d(z_2,q)\\
        &= \lVert f \rVert_L (d(p,q)-d(p,z_1)-d(z_2,q)) \geq \lVert f \rVert_L (d(z_1,z_2)-\varepsilon).\qedhere
		\end{align*}
\end{proof}

We are now ready to state the desired result.
	
	\begin{theorem}\label{teo:length}
		Let $M$ be a complete length pointed metric space. Then, $\SA(M,\R)$ is not dense in $\Lip(M,\R)$.
	\end{theorem}

	\begin{proof}
		Fix $\delta>0$, $x_0\in M\setminus\{0\}$. Let us consider a curve $$\gamma_\delta \colon [0,(1+\delta)d(0,x_0)] \longrightarrow M$$ joining $0$ and $x_0$.

		Now, let us consider a Lipschitz function  $u_0\colon \gamma_\delta([0,(1+\delta)d(0,x_0)]) \longrightarrow \mathbb{R}$ such that $u_0(0)=0$, $u_0(x_0)=1$. Since $\gamma_\delta([0,(1+\delta)d(0,x_0)])$ is compact and connected, we have that $u_0(\gamma_\delta([0,(1+\delta)d(0,x_0)]))$ is a compact connected subset of $\mathbb{R}$, i.e.\ $u_0(\gamma_\delta([0,(1+\delta)d(0,x_0)]))=[a_0,b_0]$ for certain $a_0$, $b_0 \in \mathbb{R}$. We will write
		\[ a=\frac{a_0}{\|u_0\|_L}, \quad b=\frac{b_0}{\|u_0\|_L}, \quad \frac{u_0}{\|u_0\|_L} \colon \gamma_\delta([0,(1+\delta)d(0,x_0)])\longrightarrow[a,b]. \]
		We can apply McShane's extension theorem to $\frac{u_0}{\|u_0\|_L}$ to get a surjective function $u\colon M\longrightarrow[a,b]$ verifying that  $\|u\|_L=1$. 		Let $A\subseteq [a,b]$ be a nowhere dense closed set of positive Lebesgue measure. Consider $g \in \Lip([a,b],\R)$ the function whose derivate equals $\chi_A$ (characteristic function of $A$). 		We define $h=g\circ u\colon M\longrightarrow\mathbb{R}$. It is clear that $h(0)=g(u(0))=g(0)=0$ and $\|h\|_L=\|g\|_L=1$. Therefore, $h \in \Lip(M,\R)$.
		Now, take $f \in \SA(M,\R)$. We will show that $\lVert h-f \rVert_L \geq\frac{1}{2}$. To this end, assume the contrary, that is, \ 		\[ \lVert f-h \rVert_L< \frac{1}{2}. \]
		In particular, note that $\lVert f \rVert_L >\frac{1}{2}$. We know that there exist $p,q \in M$ with $p\neq q$ such that
		\[ \lVert f \rVert_L = \frac{|f(p)-f(q)|}{d(p,q)}.\]
		Suppose that $u(p)=u(q)$, hence $h(p)=h(q)$ and we have that
		\[\lVert h-f \rVert_L \geq \frac{|(h-f)(p)-(h-f)(q)|}{d(p,q)} =\frac{|f(p)-f(q)|}{d(p,q)}=\lVert f \rVert_L > \frac{1}{2}, \]
		a contradiction. Therefore, $u(p)\neq u(q)$. We can assume that $u(p)<u(q)$ without any loss of generality. By the construction of $g$, there exist $c$, $d \in \mathbb{R}$ such that the interval $[c,d]$ is contained in $ (u(p),u(q))$ and that $g$ is constant in $[c,d]$. Take $\varepsilon_0>0$ verifying
		\begin{equation}\label{condiepslengthnodens}
        0<\varepsilon_0 < |d-c|\frac{\lVert f \rVert_L - \lVert h-f \rVert_L}{\lVert f \rVert_L}
        \end{equation}
		and a rectifiable curve $\alpha_{\varepsilon_0}$ joining $p$ and $q$ such that
		\[ \length(\alpha_{\varepsilon_0})\leq d(p,q) +\varepsilon_0. \]
		Note that such a curve exists because $M$ is a length space. Let us write $\Lambda=\alpha_{\varepsilon_0}([0,d(p,q)+\varepsilon])\subseteq M$ and observe that
		\[[c,d]\subseteq (u(p),u(q))\subseteq u(\Lambda),\]
		so there exist $\tilde{z}_1$, $\tilde{z}_2 \in \Lambda$ such that $c=u(\tilde{z}_1)$, $d=u(\tilde{z}_2)$. Moreover, we have
		\[ |d-c|=|u(\tilde{z}_2)-u(\tilde{z}_1)| \leq d(\tilde{z}_2,\tilde{z}_1).\]
		Hence, if $z_1$, $z_2$ are different points of $ \Lambda$, using Lemma \ref{lemlength} we get that
		\[ \begin{split}
        |h(z_1)-h(z_2)| & \geq |f(z_1)-f(z_2)|-\lVert h-f \rVert_L d(z_1,z_2)\\
        & \geq  \lVert f \rVert_L d(z_1,z_2) - \lVert f \rVert_L \varepsilon_0 - \lVert h-f \rVert_L d(z_1,z_2)\\ & = \left( \lVert f \rVert_L - \lVert h-f \rVert_L -\frac{\varepsilon_0 \lVert f \rVert_L}{d(z_1,z_2)}\right) d(z_1,z_2).
        \end{split}\]
		Taking $z_1=\tilde{z}_1$, $z_2=\tilde{z}_2$ and applying the above inequality, we have
		\[\begin{split}
         |h(\tilde{z}_1)-h(\tilde{z}_2)| & \geq \left( \lVert f \rVert_L - \lVert h-f \rVert_L - \frac{\varepsilon_0 \lVert f \rVert_L}{d(\tilde{z}_1,\tilde{z}_2)}\right)d(\tilde{z}_1,\tilde{z}_2)\\
         & \mathop{>}\limits^{\mbox{\eqref{condiepslengthnodens}}}
        (\lVert f \rVert_L - \lVert h-f \rVert_L - (\lVert f \rVert_L - \lVert h-f \rVert_L)) d(\tilde{z}_1,\tilde{z}_2) =0.
        \end{split}\]
This implies that $h(\tilde{z}_1)\neq h(\tilde{z}_2)$ and so $g(c)\neq g(d)$, getting a contradiction with the fact that $g$ is constant in $[c,d]$.\end{proof}

Let us now consider another negative example which can be seen as a generalisation of the fact that $\SA([0,1],\R)$ is not dense in $\Lip([0,1],\R)$. This new generalisation will allow us to produce examples of metric spaces $M$ with very different geometric and topological properties for which $\SA(M,\R)$ is still not dense in $\Lip(M,\R)$. In order to do that, we need to introduce a class of metric spaces $M$, the so-called $\R$-trees. An \emph{$\real$-tree} is a metric space $T$ satisfying:
\begin{enumerate}
	\item for any points $x$, $y \in T$, there exists a unique isometry $\phi$ from the closed interval $[0,d(x,y)]$ into $T$ such that $\phi(0)=x$ and $\phi(d(x,y))=y$. Such isometry will be denoted by $\phi_{xy}$;
	\item any one-to-one continuous mapping $\varphi\colon [0,1]\longrightarrow T$ has the same range as the isometry $\phi$ associated to the points $x=\varphi(0)$ and $y=\varphi(1)$.
\end{enumerate}
Let us introduce some notation, coming from \cite{Godard}. Given points $x,y$ in an $\real$-tree $T$, it is usual to write $[x,y]$ to denote the range of $\phi_{xy}$, which is called a segment. We say that a subset $A$ of $T$ is \emph{measurable} whenever $\phi_{xy}^{-1}(A)$ is Lebesgue measurable for any $x$, $y \in T$. If $A$ is measurable and $S$ is a segment $[x,y]$, we write $\lambda_S(A)$ for $\lambda(\phi_{xy}^{-1}(A))$, where $\lambda$ is the Lebesgue measure on $\real$. We denote by $\mathcal{R}$ the set of those subsets of $T$ which can be written as a finite union of disjoint segments, and for $R=\bigcup_{k=1}^n S_k$ (with disjoint $S_k$) in $\mathcal{R}$, we put
\[ \lambda_R(A)=\sum_{k=1}^n \lambda_{S_k}(A). \]
Now, we can define the \emph{length measure} of a measurable subset $A$ of $T$ by
\[ \lambda_T(A)= \sup_{R\in \mathcal{R}} \lambda_R(A). \]
$\mathbb R$-trees were considered in \cite{Godard} in order to characterise those metric spaces $M$ for which $\mathcal F(M)$ is isometric to a subspace of $L_1$ as those  which isometrically embed into an $\mathbb R$-tree.

Here is the promised generalisation of the fact that $\SA([0,1],\R)$ is not dense in $\Lip([0,1],\R)$.

\begin{theorem}\label{nodensiR-trees}
Let $T$ be a pointed $\real$-tree and let $M$ be a closed subset of $T$ containing the origin. If $M$ has positive length measure, then $\SA(M,\R)$ is not dense in $\Lip(M,\R)$.
\end{theorem}
	
\begin{proof}
	Note that, as $M$ has positive length measure, we can find a segment $S=[x_0,y_0] \subseteq T$ such that $\lambda_T(M\cap S) >0$.
	We distinguish two cases:\\
	First, assume that there exists a segment $[x_1,y_1]\subseteq M\cap S$. By Theorem 2.3 in \cite{kms} we know that there exists a function $f \in \Lip([x_1,y_1],\R)$ such that $\| f \|_L =1 $ and $ \| f-g\|_L \geq \frac{1}{2}$ holds for all $g \in \SA([x_1,y_1],\R)$. Consider $\pi_1 \colon T \longrightarrow [x_1,y_1]$ the metric projection, which satisfies that  \[d(x,y)=d(x,\pi_1(x))+d(\pi_1(x),y)\quad \forall \, x \in T, y \in [x_1,y_1]\]
(c.f.\ e.g.\ \cite[Chapter II.2]{bh}). Define the norm-one Lipschitz function $\tilde{f}\colon M \longrightarrow \mathbb{R}$ by $\tilde{f}(p)=[f\circ \pi_1](p) $ for every $ p \in M$, and suppose that there exists $g \in \SA(M,\R)$ such that $\| \tilde{f}-g \|_L < \frac{1}{2}$. If we take $x$, $y \in M$ satisfying that $x \neq y$ and $\langle g,m_{x,y}\rangle=\|g\|_L$, we get
	\[ \frac{1}{2} > \frac{|f(\pi_1(x))-f(\pi_1(y))- (g(x)-g(y))|}{d(x,y)}\geq \| g \|_L - \frac{|f(\pi_1(x))-f(\pi_1(y))|}{d(x,y)},\]
	so $\pi_1(x) \neq \pi_1(y)$. Using that $\langle g,m_{x,y} \rangle=\|g\|_L$,  Lemma 2.2 in \cite{kms} gives that $\langle g, m_{\pi_1(x),\pi_1(y)} \rangle=\| g \|_L$. Hence, $\restr{g}{[x_1,x_2]} \in \SA([x_1,y_1],\R)$. It follows from this that
	\[ \| \tilde{f}-g \|_L \geq \| f- \restr{g}{[x_1,y_1]} \|_L \geq \frac{1}{2},\]
	a contradiction.\\
	Now, assume that no segment is contained in $M\cap S$. Define the norm-one Lipschitz function $f \colon S \longrightarrow \mathbb{R}$ by \[f(t)=\int_{[x_0,t]} \chi_M (x)\, dx=\lambda_T([x_0,t]\cap M) \]
    for all $t \in [x_0,y_0]$. As above, define $\tilde{f}\colon M \longrightarrow \mathbb{R}$ by $\tilde{f}(p)=[f\circ \pi_2](p)$ for every $p \in M$, where $\pi_2 \colon M \longrightarrow S$ is the metric projection onto $S$.
    Again, assume that there exists $g \in \SA(M,\R)$ such that $\|g-\tilde{f}\|_L <\frac{1}{2}$. Take $x, y \in M$ such that $x\neq y$ and $\langle g, m_{x,y} \rangle = \| g \|_L$. Then, using the same argument as above, we deduce that $\pi_2(x)\neq \pi_2(y)$. Now, since $[\pi_2(x),\pi_2(y)]\nsubseteq M\cap S$ by the assumption, we can find distinct points $x_2$, $y_2 \in M$ such that $[x_2,y_2] \subseteq \, ]\pi_2(x),\pi_2(y)[ \, \setminus(M\cap S)$. Recall that $\langle g, m_{x,y} \rangle=\|g\|_L$ and this implies that $\langle g, m_{x_2,y_2} \rangle=\|g\|_L$ by Lemma 2.2 in \cite{kms}. On the other hand, note that
	\[ \tilde{f}(x_2)=f(x_2)=\lambda_T([x_0,x_2]\cap M)=\lambda_T([x_0, y_2] \cap M)=f(y_2)=\tilde{f}(y_2).\]
	Therefore, we obtain
	\[ \frac{1}{2} > \|g-\tilde{f}\|_L \geq \langle g-\tilde{f}, m_{x_2,y_2} \rangle =\langle g, m_{x_2,y_2}\rangle=\|g\|_L > \frac{1}{2},\]
	getting again a contradiction. Consequently, the set $\SA(M,\R)$ is not dense in $\Lip(M,\R)$, as desired.\end{proof}

As a particular case, we obtain the following corollary.

\begin{corollary}\label{cor[0,1]}
Let $M$ be a closed pointed subset of $[0,1]$ whose Lebesgue measure is positive. Then $\SA(M,\R)$ is not dense in $\Lip(M,\R)$.
\end{corollary}

\begin{remark}\label{remaejerarosRtree}{\slshape
Notice that the examples of metric spaces $M$ such that $\SA(M,\R)$ is not dense in $\Lip(M,\R)$ provided by Theorem \ref{teo:length} (and so by \cite[Theorem 2.3]{kms}) have very strong topological properties. For instance, it is clear that length metric spaces are arc-connected and, in particular, do not have isolated points. Nevertheless, Corollary \ref{cor[0,1]} produces quite different kind of such examples. For example, let $M$ be any nowhere dense subset of $[0,1]$ whose Lebesgue measure is positive (e.g.\ any so-called ``fat'' Cantor set). Corollary \ref{cor[0,1]} implies that $\SA(M,\R)$ is not dense in $\Lip(M,\R)$, and $M$ is totally disconnected.}
\end{remark}

As a consequence of Theorem \ref{nodensiR-trees}, we can characterise when $\SA(M,\R)$ is dense in $\Lip(M,\R)$ for compact subsets of $\mathbb R$-trees. Indeed, if $M$ is a compact subset of an $\real$-tree such that $\lambda_T(M)=0$, then $\mathcal{F}(M)$ is isometric to a subspace of $\ell_1$  \cite[Proposition~8]{dkp}, so $\mathcal{F}(M)$ has the RNP and thus, $\SA(M,\R)$ is dense in $\Lip(M,\R)$ by \cite[Proposition 7.4]{lppr} (see Theorem \ref{teornpdensidad} below). Consequently, the following corollary follows.
	
\begin{corollary}\label{caraR-tree} Let $T$ be a pointed $\real$-tree and let $M$ be a compact subset of $T$ containing $0$. Then, $\overline{\SA(M,\R)}=\Lip(M,\R)$ if, and only if, $\lambda_T(M)=0$.
\end{corollary}

\section{A discussion in Lipschitz-free spaces on sufficient conditions for Lindenstrauss property A}\label{sec:property_A}

The starting point for this section is \cite[Proposition 7.4]{lppr}, which we present here with a short sketch of a proof slightly different to the one given in \cite{lppr}. In order to do that, we need a bit of notation. Let $X$ and $Y$ be Banach spaces. We say that an operator $T\in \mathcal{L}(X,Y)$ is \emph{absolutely strongly exposing} if there exists $x \in S_{X}$ such that for every sequence $\{x_n\} \subset B_X$ such that $\lim_{n} \|Tx_n\| = \|T\|$, there is a subsequence $\{x_{n_k}\}$ which converges to either $x$ or $-x$. Clearly, if $T$ is an absolutely strongly exposing operator, then $T$ attains its norm at the point $x$ appearing at the definition; it is easy to show that such point $x\in S_X$ is a strongly exposed point (indeed, let $y^*\in S_{Y^*}$ such that $y^*(Tx)=\|T\|$ and consider $x^*\in S_{X^*}$ such that $\|T\|x^*=T^*(y^*)$; if $\{x_n\}$ is a sequence in $B_X$ such that $x^*(x_n)\longrightarrow 1=x^*(x)$, then $$\|T(x_n)\|\geq y^*(Tx_n)=\|T\|x^*(x_n)\longrightarrow \|T\|,$$ so there is a subsequence $\{x_{n_k}\}$ converging to $x$ (it cannot converge to $-x$), showing that $x$ is strongly exposed by $x^*$). A famous result of J.~Bourgain \cite[Theorem~5]{bourgain1977} says that if $X$ is a Banach space with the RNP and $Y$ is any Banach space, the set of absolutely strongly exposing operators from $X$ to $Y$ is a $G_\delta$-dense subset of $\mathcal{L}(X,Y)$ (in particular, the space $X$ has Lindenstrauss property A). Now, let $M$ be a pointed metric space such that $\mathcal{F}(M)$ has the RNP and let $Y$ be a Banach space. As $\strexp{B_{\mathcal F(M)}}\subset \Mol{M}$ (see Proposition \ref{prop:extremalidad}), the discussion above shows that the set of those elements in $\mathcal{L}(\mathcal{F}(M),Y)$ which attain their norm at points of $\Mol{M}$ is dense, in other words, $\SA(M,Y)$ is dense in $\Lip(M,Y)$.

\begin{theorem}[\mbox{\cite[Proposition 7.4]{lppr}}]\label{teornpdensidad}
Let $M$ be a complete pointed metric space such that $\mathcal F(M)$ has the RNP. Then $\SA(M,Y)$ is dense in $\Lip(M,Y)$ for every Banach space $Y$.
\end{theorem}

Roughly speaking, the proof of the previous theorem shows how a property of Banach spaces (the RNP) which implies property A may actually imply the density of $\SA(M,Y)$ in $\Lip(M,Y)$ by making a strong use of the special behaviour of the extremal structure of Lipschitz-free spaces. This fact motivates an analysis of the connections between certain linear properties on $\mathcal F(M)$ which imply property A and the fact that $\SA(M,Y)$ is dense in $\Lip(M,Y)$ for every Banach space $Y$.

The properties implying property A that we will discuss in the setting of Lipschitz-free spaces will be the following ones, whose definitions can be found in the respective subsections:
\begin{itemize}
\item the existence of a set of uniformly strongly exposed points whose closed convex hull equals the unit ball, introduced by Lindenstrauss himself \cite{lindens} in 1963;
\item property $\alpha$, introduced by W.~Schachermayer \cite{Schachermayer} in 1983, which implies the previous one and which satisfies that ``many'' Banach spaces (separable, reflexive, WCG\ldots) can be renormed having it;
\item property quasi-$\alpha$, which is weaker than property alpha but still implies property A, introduced by Y.~S.~Choi and H.~G.~Song \cite{ChoiSong} in 2008.
\end{itemize}
\begin{figure}[h]
\centering
  \begin{tikzcd}[column sep=1.8cm]
  \lfbox{Property $\alpha$}
  \arrow[d, Rightarrow] & \lfbox{Property quasi-$\alpha$} \arrow[l, Rightarrow]\arrow[d, Rightarrow]& \\
  \lfbox{$\begin{array}{c} B_X=\cconv(S) \\  S \text{ unif. str. exp.}\end{array}$}\arrow[r, Rightarrow] & \lfbox{Property A} & \lfbox{RNP} \arrow[l, Rightarrow]
     \end{tikzcd}
\caption{Relations between properties implying property A in general Banach spaces}
\label{figure:general}
\end{figure}
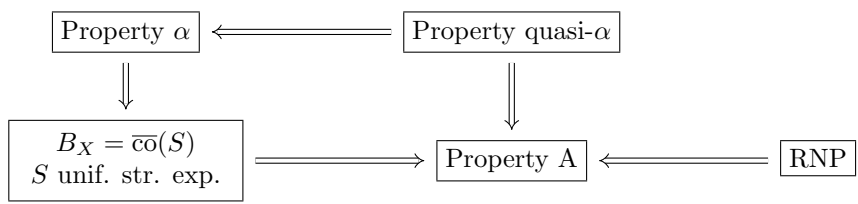
In general, given a Banach space $X$, we have the implications given in Figure \ref{figure:general}. None of the implications reverses and the RNP is not related to the others three properties implying property A. We will discuss the relations between the properties in the setting of Lipschitz-free spaces in subsection \ref{subsec:relations}.

\subsection{Uniformly strongly exposed points}
We start with the definition of the property.

\begin{definition}
Let $X$ be a Banach space. A subset $S\subset S_X$ is said to be a \emph{set of uniformly strongly exposed points} if there is a family of functionals $\{h_x\}_{x\in S}$ with $\|h_x\|=h_x(x)=1$ for every $x\in S$ such that, given $\varepsilon>0$ there is $\delta>0$ satisfying that
\[
\sup_{x\in S} \diam(S(B_X,h_x, \delta)) \leq \varepsilon;
\]
equivalently, if for every $\varepsilon>0$ there is $\delta'>0$ such that whenever $z\in B_X$ satisfies $h_x(z)>1-\delta'$ for some $x\in S$, then $\|x-z\|<\varepsilon$ (that is, all elements of $S$ are strongly exposed points with the same relation $\varepsilon$--$\delta$).
\end{definition}

This concept appeared in the seminal paper by Lindenstrauss \cite{lindens} (see also \cite{Finet} for further applications of it) to give a sufficient condition for a Banach space $X$ to enjoy property A. Namely, if $X$ is a Banach space containing a set of uniformly strongly exposed points $S\subset S_X$ such that $B_X=\overline{\co}(S)$, then $X$ has property A \cite[Proposition 1]{lindens}. Moreover, having a look at the proof of the result, something more can be said. In fact, it is actually proved that, given a Banach space $Y$, then the set
$$
\left\{T\in \mathcal{L}(X,Y)\colon T\text{ attains its norm at a point of }\overline{S}\right\}
$$
is dense in $\mathcal{L}(X,Y)$. Now, given a metric space $M$, if $\mathcal F(M)$ has a subset $S\subseteq S_X$ of uniformly strongly exposed points, then $S\subseteq \Mol{M}$ by Proposition \ref{prop:extremalidad}, since $S$ is made of strongly exposed points. Now, as $\Mol{M}$ is norm-closed (use Proposition \ref{prop:extremalidad} again), the following result follows.

\begin{proposition}\label{prop:unifstrexpsnadens}
Let $M$ be a complete pointed metric space and assume that $B_{\mathcal F(M)}$ is the closed convex hull of a set of uniformly strongly exposed points. Then $\SA(M,Y)$ is norm dense in $\Lip(M,Y)$ for every Banach space $Y$.
\end{proposition}

In view of the previous proposition, we shall begin with a characterisation, inspired by \cite[Theorem 5.4]{gpr}, of the existence of such a set of uniformly strongly exposed points, which only depends on the metric space $M$. We previously need a technical lemma.

\begin{lemma}\label{lemma:failZuniformly}
Let $M$ be a complete pointed metric space, let $A = \{m_{x,y}\}_{(x,y)\in \Lambda}$ be a family of molecules in $\mathcal F(M)$. Assume that there is $\varepsilon_0>0$ such that
	\[ (x,y)_z > \varepsilon_0 \min\{d(x,z), d(y,z)\} \]
	whenever $m_{x,y}\in A$ and $z\in M\setminus\{x,y\}$. Then, there exists a family $B = \{h_{x,y}\}_{(x,y)\in \Lambda}$ in $S_{\Lip(M,\R)}$ such that
	\begin{itemize}
		\item[(a)] $\langle h_{x,y}, m_{x,y}\rangle  = 1$ for every $(x,y)\in \Lambda$, and
		\item[(b)] for every $\varepsilon>0$ there is $\delta = \delta(\varepsilon, \varepsilon_0)>0$ such that
		\begin{equation}\label{eq:lemmab}
		 \langle h_{x,y}, m_{u,v}\rangle > 1-\delta \ \ \text{ implies } \ \ \norm{m_{x,y} - m_{u,v}}<\varepsilon
		\end{equation}
		for every $(x,y)\in \Lambda$ and every $u,v \in M$, $u\neq v$.
	\end{itemize}
\end{lemma}

\begin{proof}
Fix $\varepsilon_1>0$ with $\frac{\varepsilon_1}{1-\varepsilon_1}<\frac{\varepsilon_0}{4}$. For $x,y\in M$ such that $m_{x,y}$ belongs to $A$, consider the Lipschitz function $g_{x,y}$ defined in~\cite[Proposition 2.8]{ikw}, namely
\[ g_{x,y}(z):=\begin{cases}
\max\left\{\frac{d(x,y)}{2}-(1-\varepsilon_1)d(z,x), 0\right\} & \text{if } d(z,y)\geq d(z,x),\\
& d(z,y)+(1-2\varepsilon_1)d(z,x)\geq d(x,y), \\
-\max\left\{\frac{d(x,y)}{2}-(1-\varepsilon_1)d(z,y), 0\right\} & \text{if } d(z,x)\geq d(z,y),\\
& d(z,x)+(1-2\varepsilon_1)d(z,y)\geq d(x,y).
\end{cases}
\]
It is well defined and satisfies that $\norm{g_{x,y}}_{L}=1$, $\langle g_{x,y}, m_{x,y}\rangle  = 1$, and
\begin{equation}\label{eq:gxy}
\langle g_{x,y}, m_{u,v}\rangle  >1-\varepsilon_1\ \  \text{ implies } \ \ \max\{d(x,u),d(y,v)\}<\frac{d(x,y)}{4}
\end{equation}
for any $u,v\in M$, $u\neq v$ (see the proof of Proposition 2.8 in~\cite{ikw}). Consider also the function defined by
\[f_{x,y}(t):= \frac{d(x,y)}{2}\frac{d(t,y)-d(t,x)}{d(t,y)+d(t,x)}\] for every $t\in M$,
and take $h_{x,y}=\frac{1}{2}(g_{x,y}+f_{x,y})$. Now, one can check that the family $B=\{h_{x,y}\}_{(x,y)\in \Lambda}$ does the work following word-by-word the proof of \cite[Theorem 5.4]{gpr}.
\end{proof}

The previous lemma motivates us to consider the following property, related to the  characterization of strongly exposed points of the unit ball of the Lipschitz-free spaces given in Proposition \ref{prop:extremalidad}.b.

\begin{definition} Let $M$ be a metric space and let $A\subset \Mol{M}$. We say that \emph{$A$ is uniformly Gromov rotund} if there is $\varepsilon>0$ such that
	\[ (x,y)_z > \varepsilon \min\{d(x,z), d(y,z)\} \]
	whenever $m_{x,y}\in A$ and $z\in M\setminus\{x,y\}$.
\end{definition}

Observe that, thanks to Proposition \ref{prop:extremalidad}.b, $A$ is uniformly Gromov rotund if and only if it fails property (Z) in a uniform way.

We can now give a metric characterisation of when a set of molecules in $\mathcal F(M)$ is uniformly strongly exposed in the following sense.

\begin{proposition}\label{prop:charunifstrexp} Let $M$ be a complete pointed metric space and let $A$ be a set of molecules in $\mathcal F(M)$. Then, the following statements are equivalent:
\begin{enumerate}[(i)]
\item $A$ is uniformly strongly exposed in $B_{\mathcal F(M)}$,
\item $A$ is uniformly Gromov rotund.
\end{enumerate}
\end{proposition}

In order to prove the proposition, we need the following lemma.

\begin{lemma}\label{lemma:strexp} Let $M$ be a complete pointed metric space. Let $x,y\in M$, $x\neq y$, and let $f\in \Lip(M,\R)$ be such that $\norm{f}_L=1$ and $\langle f, m_{x,y}\rangle  =1$. Then, for every $z\in M\setminus\{x,y\}$ we have that
\[  \langle f, m_{x,z}\rangle  \geq 1-2\frac{(x,y)_z}{d(x,z)} \quad \text{and} \quad \langle f, m_{z,y}\rangle  \geq 1-2\frac{(x,y)_z}{d(y,z)}.\]
\end{lemma}

\begin{proof}
Note that
\begin{align*}
1 = \langle  f, m_{x,y}\rangle   & = \left\langle f, \frac{d(x,z)}{d(x,y)}m_{x,z} + \frac{d(z,y)}{d(x,y)}m_{z,y}\right\rangle  \\ & = \frac{d(x,z)}{d(x,y)} \langle f, m_{x,z}\rangle  + \frac{d(z,y)}{d(x,y)} \langle f, m_{z,y}\rangle .
\end{align*}
Thus,
\begin{align*} d(x,z)+d(z,y)-2(x,y)_z &= d(x,y) = d(x,z)\langle f, m_{x,z} \rangle  + d(z,y) \langle  f, m_{z,y}\rangle \\  &\leq d(x,z)\langle f, m_{x,z}\rangle  + d(z,y)
\end{align*}
and the conclusion follows.
\end{proof}

\begin{proof}[Proof of Proposition \ref{prop:charunifstrexp}]
(i)$\Rightarrow$(ii). Let $\{h_{x,y}\}_{m_{x,y}\in A}\subset S_{\Lip(M,\R)}$ be a family which uniformly strongly exposes the family $A$. Take $\delta>0$ such that
\[ \sup_{m_{x,y}\in A} \diam(S(B_{\mathcal{F}(M)}, h_{x,y}, \delta)) < \frac{1}{2}.\]
Assume that $A$ is not uniformly Gromov rotund. Then, there are $x,y\in M$, $x\neq y$, such that $m_{x,y}\in A$, and $z\in M\setminus\{x,y\}$ such that
\[ (x,y)_z < \frac{\delta}{2} \min\{d(x, z), d(y, z)\}.\]
By interchanging the roles of $x$ and $y$ if needed, we may assume that $d(x, z)\leq d(y,z)$ and so, $d(y,z)\geq \frac{1}{2}d(x,y)$. Now, Lemma \ref{lemma:strexp} implies that
\[ \langle h_{x,y}, m_{x,z}\rangle  \geq 1-2\frac{(x,y)_{z}}{d(x,z)} > 1-\delta. \]
From this and Lemma \ref{lemma:ineqmolec}, it follows that
\[ \frac{1}{2}\leq\frac{d(y,z)}{d(x,y)} \leq \norm{m_{x,y}-m_{x,z}} < \frac{1}{2}\]
which is a contradiction.

\noindent (ii)$\Rightarrow$(i). By hypothesis, there is $\varepsilon_0>0$ such that
\[ d(x,z)+d(z,y)>d(x,y)+\varepsilon_0 \min\{d(x,z),d(z,y)\} \]
whenever $m_{x,y}\in A$ and $z\in M\setminus\{x,y\}$. Let $B=\{h_{x,y}\}_{(x,y)\in \Lambda}$ be the set provided by Lemma \ref{lemma:failZuniformly}. We claim that $B$ uniformly strongly exposes $A$. Indeed, given $\varepsilon>0$, take $0<\delta<\varepsilon$ such that
 \[ \langle h_{x,y}, m_{u,v}\rangle > 1-\delta\ \ \text{ implies }\ \ \norm{m_{x,y} - m_{u,v}}<\varepsilon \]
 for every $m_{x,y}\in A$ and every $u,v \in M$, $u\neq v$. Thus,
 \[ \diam\bigl(S(B_{\mathcal F(M)}, h_{x,y}, \delta)\cap \Mol{M}\bigr) \leq 2\varepsilon. \]
 Finally, note that
\[ \diam(S(B_{\mathcal F(M)}, h_{x,y},  \delta^2) \leq 2\diam\left(S(B_{\mathcal F(M)}, h_{x,y}, \delta)\cap \Mol{M}\right) + 4\delta \leq 8\varepsilon, \]
see e.g.\ Lemma 2.7 in \cite{lppr}.
\end{proof}

As an immediate consequence of Propositions \ref{prop:unifstrexpsnadens} and \ref{prop:charunifstrexp}, we obtain the following corollary.

\begin{corollary}
Let $M$ be a complete pointed metric space. If there exists a uniformly Gromov rotund subset $A\subset \Mol{M}$ such that $B_{\mathcal{F}(M)}$ is the closed convex hull of $A$, then $\SA(M,Y)$ is dense in $\Lip(M,Y)$ for every Banach space $Y$.
\end{corollary}

The space $\Lip(M,\R)$ has the Daugavet property if and only if the complete pointed metric space $M$ has property (Z) (see \cite{ikw} for the compact case and the very recent paper \cite{AviMar} for the general case) and if and only if $M$ is a length space \cite[Theorem 3.5]{gpr}. Thus, the previous result shows that the failure of the Daugavet property in a very strong sense implies the density of $\SA(M, Y)$ in $\Lip(M,Y)$ for every Banach space $Y$. Compare with Theorem \ref{teo:length}, where it is shown that if $\Lip(M,\R)$ has the Daugavet property, then $\SA(M,\R)$ is not dense in $\Lip(M,\R)$.

\subsection[Property alpha]{Property \texorpdfstring{$\alpha$}{alpha}}

In the sequel we will consider a particular way in which a Banach space may contain a subset of uniformly strongly exposed points whose closed convex hull is the whole unit ball. It was introduced in \cite{Schachermayer} by W.~Schachermayer with the name of property $\alpha$ and its main interest is that ``many'' Banach spaces (e.g.\ separable, reflexive, WCG\ldots) can be equivalently renormed to have it. The prototype Banach space with property $\alpha$ is $\ell_1$.

\begin{definition}
	A Banach space $X$ is said to have \emph{property $\alpha$} if there exist a balanced subset $\{x_\lambda\}_{\lambda \in \Lambda}$ of $X$ and a subset $\{x^*_\lambda\}_{\lambda \in \Lambda} \subseteq X^*$ such that
	\begin{enumerate}[(i)]
		\item $\lVert x_\lambda \rVert = \lVert x^*_\lambda\rVert = \lvert x^*_\lambda(x_\lambda)\rvert =1$ for all $\lambda\in \Lambda$.
		\item There exists $0\leq \rho <1$ such that
		\[ |x^*_\lambda(x_\mu)|\leq \rho \quad \forall \, x_\lambda \neq \pm x_\mu. \]
		\item $\overline{\co}\left(\{x_\lambda\}_{\lambda \in \Lambda}\right)= B_X$.
	\end{enumerate}	
\end{definition}

Because of methodological reasons, we have modified a little bit the original definition from \cite{Schachermayer} to an equivalent one in which we impose the set $\{x_\lambda\}_{\lambda \in \Lambda}$ to be balanced.

It is shown in \cite[Fact in p.~202]{Schachermayer} that if $X$ has property $\alpha$ witnessed by a set $\Gamma \subset S_X$, then $\Gamma$ is a set of uniformly strongly exposed points. Therefore, if $M$ is a pointed metric space for which $\mathcal{F}(M)$ has property $\alpha$, then Proposition \ref{prop:unifstrexpsnadens} gives that $\SA(M,Y)$ is dense in $\Lip(M,Y)$ for every Banach space $Y$. This can be also proved directly by adapting the proof of \cite[Proposition 1.3.a]{Schachermayer} to our case, as it is shown there that the set of operators from $\mathcal{F}(M)$ to $Y$ attaining their norms on points of $\Gamma$ is dense in $\mathcal{L}(\mathcal{F}(M),Y)$, and we just have to observe that $\Gamma\subset \strexp{B_{\mathcal F(M)}}\subset \Mol{M}$ (see Proposition \ref{prop:extremalidad}).

\begin{corollary}
Let $M$ be a complete pointed metric space for which $\mathcal{F}(M)$ has property $\alpha$. Then, the set $\SA(M,Y)$ is dense in $\Lip(M,Y)$ for every Banach space $Y$.
\end{corollary}

As we have said, if $\Gamma\subset S_{\mathcal F(M)}$ witnesses that $\mathcal F(M)$ has property $\alpha$, then $\Gamma$ is made up of molecules. We can say something more. J.~P.~Moreno proved in \cite[Proposition~3.6]{Moreno} that if a Banach space $X$ has property $\alpha$ witnessed by $\Gamma\subset S_X$, then $\Gamma = \dent{B_X} = \strexp{B_X}$. Indeed, if $x\in S_X$ is a denting point, then the  slices of $B_{X}$  containing $x$ are a neighbourhood basis of $x$ for the norm topology in $B_X$. Since $\cconv(\Gamma)=B_X$, we have that every slice of $B_X$ intersects $\Gamma$. It follows that $x\in \overline{\Gamma}$. Finally, if $X$ has property $\alpha$, then the set $\Gamma$ is obviously uniformly discrete, hence closed. Thus,
\[ \dent{B_X}\subset \Gamma \subset \strexp{B_X} \subset \dent{B_X} \]
From this and the fact that every preserved extreme point of $B_{\mathcal{F}(M)}$ is a denting point by Proposition~\ref{prop:extremalidad}, we get the following result.

\begin{proposition}\label{prop:alphafreestrexp} Let $M$ be a complete pointed metric space and assume that $\mathcal F(M)$ has property $\alpha$ witnessed by $\Gamma\subset S_{\mathcal F(M)}$. Then, $$\Gamma=\preext{B_{\mathcal F(M)}}=\strexp{B_{\mathcal F(M)}}.$$
\end{proposition}

In the sequel, we will get a reformulation of property $\alpha$ in $\mathcal F(M)$. To this end, we need the following elementary characterisation of uniformly discrete subsets of molecules.

\begin{lemma}\label{lemma:unifdiscrmolec} Let $M$ be a metric space and consider $A\subset \Mol{M}$. Then, $A$ is uniformly discrete if and only if there exists $\delta>0$ such that
\begin{equation}\label{eq:unifdisc}
d(x,u)+d(v,y) \geq \delta\, d(x,y)
\end{equation}
whenever $m_{x,y}$ and $m_{u,v}$ are distinct elements of $A$.
\end{lemma}

\begin{proof}
If $A$ is uniformly discrete, then there is $\delta>0$ such that
\[ 2\delta\leq \norm{m_{x,y}-m_{u,v}}\leq 2\frac{d(x,u)+d(y,v)}{d(x,y)},\]
where the last inequality follows from Lemma \ref{lemma:ineqmolec}. Conversely, assume that the inequality \eqref{eq:unifdisc} holds for every $m_{x,y}, m_{u,v}\in A$ with $m_{x,y}\neq m_{u,v}$. If one has that $\norm{m_{x,y}-m_{u,v}}<1$ then, again by Lemma \ref{lemma:ineqmolec}, we get that
\[\norm{m_{x,y}-m_{u,v}}\geq \frac{\max\{d(x,u),d(u,v)\}}{d(x,y)}\geq \frac{1}{2}\frac{d(x,u)+d(u,y)}{d(x,y)} \geq \frac{\delta}{2}.\]
Thus, $\norm{m_{x,y}-m_{u,v}}\geq \min\{1,\delta/2\}$ for $m_{x,y}\neq m_{u,v}$ in $A$.
\end{proof}

The following proposition characterises the Lipschitz-free spaces with property $\alpha$ in terms of the existence of a norming subset of molecules satisfying certain metrical conditions.

\begin{proposition}\label{prop:charalpha} Let $M$ be a complete pointed metric space. The following are equivalent:
	\begin{enumerate}[(i)]
		\item $\mathcal F(M)$ has property $\alpha$.
		\item There exists $\Lambda\subset \{(p,q)\in M\times M\colon p\neq q\}$ such that, writing $A = \{m_{x,y}\colon (x,y)\in \Lambda\}\subset \Mol{M}$, one has that:
        \begin{itemize}
        \item there exists $\delta>0$ such that $d(x,u)+d(y,v)\geq \delta d(x,y)$ for all $(x,y),(u,v)\in \Lambda$ with $(x,y)\neq (u,v)$ (equivalently, $A$ is uniformly discrete);
        \item there is $\varepsilon>0$ such that
	\[ (x,y)_z > \varepsilon \min\{d(x,z), d(y,z)\} \]
	whenever $(x,y)\in \Lambda$ and $z\in M\setminus\{x,y\}$ (equivalently, $A$ is uniformly Gromov rotund);
        \item $\|f\|_L=\sup\left\{\frac{f(x)-f(y)}{d(x,y)}\colon (x,y)\in \Lambda\right\}$  for every $f\in \Lip(M,\R)$ (equivalently, $B_{\mathcal F(M)} = \cconv(A)$).
 	\end{itemize}
	\end{enumerate}
Moreover, in such a case, the set $A$ coincides with the whole set of strongly exposed points of $B_{\mathcal F(M)}$.
\end{proposition}

\begin{proof}
	(i)$\Rightarrow$(ii). Let $A\subset S_{\mathcal F(M)}$ witnessing that $\mathcal F(M)$ has property $\alpha$. Then $B_{\mathcal F(M)} = \cconv(A)$. Moreover, it is clear that $A$ is uniformly discrete and it is known that it is uniformly strongly exposed \cite[Fact in p.~202]{Schachermayer}, so Proposition~\ref{prop:charunifstrexp} and Lemma \ref{lemma:unifdiscrmolec} give the result.
	
	(ii)$\Rightarrow$(i). Let $A=\{m_{x,y}\}_{(x,y)\in \Lambda}$ be a set of molecules satisfying the properties in the statement. Let $B=\{h_{x,y}\}_{(x,y)}\subset S_{\Lip(M,\R)}$ be the family provided by Lemma~\ref{lemma:failZuniformly}. By Lemma~\ref{lemma:unifdiscrmolec},
    \[ \varepsilon = \inf\{\norm{m_{x,y}-m_{u,v}}\colon m_{x,y}, m_{u,v}\in A,\, m_{x,y}\neq m_{u,v}\}>0.\]
    Take $\delta>0$ such that \eqref{eq:lemmab} in Lemma \ref{lemma:failZuniformly} holds for that $\varepsilon$. Then, $$\langle h_{x,y}, m_{u,v}\rangle  \leq 1-\delta$$ whenever $m_{x,y}, m_{u,v}\in A$ and $m_{x,y}\neq \pm m_{u,v}$. Thus, $\mathcal F(M)$ has property $\alpha$.

    The last assertion follows from Proposition~\ref{prop:alphafreestrexp}.
\end{proof}

We can provide an easier characterisation in the bounded and uniformly discrete case.

\begin{proposition}\label{prop:caralphaudiscreto} Let $M$ be a pointed bounded and uniformly discrete metric space. The following are equivalent:
	\begin{enumerate}[(i)]
		\item $\mathcal F(M)$ has property $\alpha$.
        \item the set $\strexp{B_{\mathcal F(M)}}$ consists of uniformly strongly exposed points (equivalently, it is uniformly Gromov rotund).
		\item there is $\varepsilon >0$ such that for every $x,y\in M$ with $x\neq y$, \[\text{either}\quad \inf_{z\in M\setminus \{x,y\}} (x,y)_z = 0 \quad \text{or} \quad \inf_{z\in M\setminus\{x,y\}} (x,y)_z \geq \varepsilon.\]
	\end{enumerate}
\end{proposition}

\begin{proof} Denote $D = \sup\{d(x,y)\colon x\neq y\}<\infty$ and $\theta = \inf\{d(x,y)\colon x\neq y\}>0$.

(i)$\Rightarrow$(ii) follows from Propositions \ref{prop:alphafreestrexp} and \ref{prop:charalpha}.

Next, assume that (ii) holds. Then, there is $\varepsilon>0$ such that
\[ (x,y)_z \geq \varepsilon \min\{d(x,z),d(z,y)\}\geq \varepsilon \theta\]
whenever $m_{x,y}\in \strexp{B_{\mathcal F(M)}}$. So, given $x, y\in M$, $x\neq y$, either $m_{x,y}$ is strongly exposed, and then $\inf_{z\in M\setminus\{x,y\}} (x,y)_z \geq \varepsilon\theta$, or $m_{x,y}$ is not strongly exposed, and then
\[ \inf_{z\in M\setminus\{x,y\}} (x,y)_z \leq D\inf_{z\in M\setminus\{x,y\}} \frac{(x,y)_z}{\min\{d(x,z),d(y,z)\}} =0.\]
This gives (iii).

Finally, assume that (iii) holds and let $A=\strexp{B_{\mathcal F(M)}}$. Then, for every $m_{x,y}\in A$ we have that $\inf_{z\in M\setminus \{x,y\}} (x,y)_z > 0$ and so \[\inf_{z\in M\setminus \{x,y\}} (x,y)_z \geq \varepsilon\geq \frac{\varepsilon}{D}\min\{d(x,z),d(z,y)\}.\]
That is, $A$ is uniformly Gromov rotund. Moreover, $B_{\mathcal F(M)}=\cconv(A)$ since $\mathcal F(M)$ has the RNP. Finally,
\[ d(x,u)+d(v,y)\geq \delta d(x,y),\]
for every distinct pairs of points $(x,y), (u,v) \in \{(p,q)\in M\times M\colon p\neq q\}$, where $\delta = 2\theta/D$. By Proposition~\ref{prop:charalpha}, $\mathcal F(M)$ has property $\alpha$, getting (i).
\end{proof}

Let us exhibit some examples of metric spaces such that $\mathcal F(M)$ has property $\alpha$.

\begin{example}\label{ex:spacesAlpha} The space $\mathcal F(M)$ has property $\alpha$ in the following cases:
\begin{enumerate}[(a)]
\item $M$ is finite.
\item  $M$ is a compact subset of $\mathbb R$ with measure $0$.
\item There exists a constant $1\leq D<2$ such that
$$1\leq d(x,y)<D$$
holds for every pair of distinct points $x,y\in M$ (equivalently, up to rescaling, there are constants $C>0$ and $1\leq D < 2$ such that $C\leq d(x,y)<CD$ for all $x,y\in M$, $x\neq y$).
\end{enumerate}
\end{example}

\begin{proof}
(a). Given $m_{x,y}\in \strexp{B_{\mathcal F(M)}}$, consider a strongly exposing functional $g_{x,y}\in S_{\Lip(M,\R)}$. Take $\rho$ to be the maximum of the set
\[
\{|\langle g_{x,y}, m_{u,v}\rangle | \colon m_{x,y}\in \strexp{B_{\mathcal F(M)}}, m_{u,v}\in \Mol{M}, m_{x,y}\neq \pm m_{u,v}\}.\]
Then, $\rho<1$ since $M$ is finite. Moreover, $\mathcal F(M)$ is finite dimensional and so $B_{\mathcal F(M)}$ is the closed convex hull of its strongly exposed points. Thus, $\mathcal F(M)$ has property $\alpha$.

\noindent (b). $\mathcal F(M)$ is isometric to $\ell_1$ by \cite{Godard}, so it clearly has property $\alpha$.

\noindent (c). Let $0<\varepsilon<\frac{2}{D}-1$. Observe that given $x,y,z\in M$, we get
\[\begin{split}
\varepsilon \min\{d(x,z),d(y,z)\}& \leq \varepsilon D<2-D\leq d(x,z)+d(y,z)-D\\
& \leq d(x,z)+d(y,z)-d(x,y)=2(x,y)_z.
\end{split}\]
Consequently, if we define $\Lambda:=\{(p,q)\in M\times M\colon p\neq q\}$, then $\Lambda$ satisfies the condition (ii) in Proposition \ref{prop:charalpha}, and so $\mathcal F(M)$ has property $\alpha$.
\end{proof}

The next result provides a characterisation of those concave metric spaces for which $\mathcal F(M)$ has property $\alpha$.

\begin{theorem}\label{alpha+unifconc}
	Let $M$ be a complete pointed concave metric space. Then the following are equivalent:
    \begin{enumerate}[(i)]
    \item $\mathcal F(M)$ has property $\alpha$.
    \item $M$ is uniformly discrete and bounded, and there is $\varepsilon>0$ such that
    \[ d(x,z)+d(z,y)-d(x,y)\geq \varepsilon\]
whenever $x,y,z$ are distinct points in $M$.
\end{enumerate}
\end{theorem}

\begin{proof}
Assume first that $\mathcal F(M)$ has property $\alpha$ with constant $\rho>0$. By Proposition~\ref{prop:alphafreestrexp}, the set $\Gamma\subset S_{\mathcal{F}(M)}$ witnessing property $\alpha$ coincides with $\preext{B_{\mathcal F(M)}}$, so $\Gamma=\Mol{M}$ as $M$ is concave. Now, take $m_{x,y},m_{u,y}\in \Mol{M}=\Gamma$ and let $g_{x,y}\in S_{\Lip(M,\R)}$ be the functional associated to $m_{x,y}$. Then, by Lemma \ref{lemma:ineqmolec}, we have that
\begin{equation*}
2\frac{d(x,u)}{d(x,y)}\geq \norm{m_{x,y}-m_{u,y}}\geq |g_{x,y}(m_{x,y}-m_{u,y})|\geq 1-\rho.
\end{equation*}
From here, given $x,u\in M$ we have that
\[(1-\rho)\sup\limits_{y\in M} d(x,y)\leq 2d(x,u),\]
from where it follows that $M$ is bounded. Moreover, the following estimate holds:
\[(1-\rho)\diam(M)\leq 2(1-\rho)\sup\limits_{y\in M} d(x,y)\leq 4d(x,u).\]
Since $x,u\in M$ were arbitrary we conclude that $M$ is uniformly discrete. Now, Proposition~\ref{prop:caralphaudiscreto} provides $\varepsilon>0$ such that $(x,y)_z\geq \varepsilon$ whenever $m_{x,y}\in \strexp{B_{\mathcal F(M)}}$ and $z\in M\setminus\{x,y\}$. Since every molecule is strongly exposed, the conclusion follows.

Finally, the converse statement follows from Proposition~\ref{prop:caralphaudiscreto}.
\end{proof}

As an application of the previous theorem, we may show that $D=2$ is not possible in Example \ref{ex:spacesAlpha}.c.

\begin{example}
Let $M=\{0,x_n,y_n\colon n\geq 2\}\subseteq c_0$, where $x_n:=(2-\frac{1}{n}) e_n$ and $y_n:=e_n+(1+\frac{1}{n})e_1$ for every $n\geq 2$. It can be proved routinely that $M$ is concave by using the characterisation of the preserved extreme points given in \cite{ag}. On the other hand, it is clear that the inequality $$1\leq d(x,y)<2$$ holds for every $x,y\in M$ with $x\neq y$. Nevertheless, one has that $$d(0,y_n)+d(y_n,x_n)-d(0,x_n)=\frac{3}{n}$$ for every $n\geq 2$, so $\mathcal F(M)$ fails property $\alpha$ by Theorem \ref{alpha+unifconc}.
\end{example}

\subsection[Property quasi-alpha]{Property quasi-\texorpdfstring{$\alpha$}{alpha}}
In \cite{ChoiSong} it is defined a property which, in spite of being weaker than property $\alpha$, still implies property $A$. As in the case of property $\alpha$, we have slightly modified the original definition to an equivalent one which requires the set $\{x_\lambda\}_{\lambda \in \Lambda} \subseteq X$ bellow to be balanced.

\begin{definition}
	A Banach space $X$ is said to have \textit{property quasi-$\alpha$} if there exist a balanced subset $A:=\{x_\lambda\}_{\lambda \in \Lambda}$ of $X$, a subset $\{x_\lambda^*\}_{\lambda \in \Lambda}\subseteq X^*$, and $\rho \colon \Lambda \longrightarrow \mathbb{R}$ such that
	\begin{itemize}
		\item[a)] $\| x_\lambda \|= \| x^*_\lambda\| =\| x^*_\lambda (x_\lambda)\|=1 $ for all $\lambda \in \Lambda$.
		\item[b)] $|x^*_\lambda(x_\mu)| \leq \rho(\mu)<1$ for all $x_\lambda \neq \pm x_\mu$.
		\item[c)] For every $e \in \ext{B_{X^{**}}}$, there exists a subset $A_e \subseteq A$ such that either $e$ or $-e$ belong to $\overline{A_e}^{\omega^*}$ and $r_e=\sup\{\rho(\mu)\colon x_\mu \in A_e\}<1$.
	\end{itemize}
\end{definition}

It follows that if $\{x_\lambda\}_{\lambda \in \Lambda}$ witnesses that $X$ has property quasi-$\alpha$, then $$B_X=\overline{\co}(\{x_\lambda\colon \lambda \in \Lambda\}).$$ Moreover, the same argument as the one used for property $\alpha$ in \cite[Fact in p.~202]{Schachermayer}, shows that for every $\lambda \in \Lambda$, $\varepsilon >0$, and $x \in B_X$, one has that
\[ x^*_\lambda(x)> 1-\varepsilon(1-\rho(\lambda)) \Longrightarrow \| x-x_\lambda\| < 2 \varepsilon;\]
so each $x_\lambda$ is strongly exposed in $B_X$ by $x^*_\lambda$. But now, as $\sup_{\lambda\in \Lambda}\rho(\lambda)$ may be equal to one, we do not get that $\{x_\lambda\}_{\lambda \in \Lambda}$ is a set of uniformly strongly exposed points. Nevertheless, the proof of Proposition 2.1 in \cite{ChoiSong} shows that if $\mathcal{F}(M)$ has property quasi-$\alpha$ then the set
\[
\mathcal{A}:=\left\{T \in \mathcal{L}(\mathcal{F}(M), Y) \colon \| T \|=\|T(x_\lambda)\| \text{ for some } \lambda \in \Lambda \right\}
\]
is norm-dense in $\mathcal{L}(\mathcal{F}(M), Y) \equiv \Lip(M,Y)$. Now, every $x_\lambda$ is a strongly exposed point of $B_{\mathcal{F}(M)}$, and so, a molecule by Proposition \ref{prop:extremalidad}. Thus, $\mathcal{A}\subseteq \SA(M,Y)$. We have proved the following.

\begin{proposition}\label{prop:qalphaimplidensity}
	Let $M$ be a complete pointed metric space and assume that $\mathcal{F}(M)$ has property quasi-$\alpha$. Then $\SA(M,Y)$ is norm-dense in $\Lip(M,Y)$ for every Banach space $Y$.
\end{proposition}

An analogous argument to the one given in the proof of Proposition~\ref{prop:alphafreestrexp} shows the following:

\begin{proposition}\label{prop:qalphafreestrexp} Let $M$ be a complete pointed metric space and assume that $\mathcal F(M)$ has property quasi-$\alpha$ witnessed by a set $\Gamma\subset S_{\mathcal F(M)}$. Then, $$\preext{B_{\mathcal F(M)}}\subset \overline{\Gamma}.$$
\end{proposition}

As a consequence, we obtain the following result in the case when $M$ is concave.

\begin{proposition}\label{qalpha+unifconc}
Let $M$ be a concave complete pointed metric space. If $\mathcal F(M)$ has property quasi-$\alpha$, then the set of isolated points of $M$ is dense in $M$.
\end{proposition}

\begin{proof}
Assume that $\mathcal F(M)$ has property quasi-$\alpha$ witnessed by the sets $\Gamma\subset S_{\mathcal F(M)}$, $\Gamma^*\subset S_{\Lip(M,\R)}$, and the function $\rho\colon \Gamma\longrightarrow \mathbb R$. Take $m_{x,y}\in \Gamma$ and let $g_{x,y}\in \Gamma^*$ be its associated functional. Then,
\begin{equation}\label{eq:qalpha}
\norm{m_{x,y}-m_{u,v}}\geq |g_{x,y}(m_{x,y}-m_{u,v})|\geq 1-\rho(m_{x,y})
\end{equation}
for every $m_{u,v}\in \Gamma$ with $m_{u,v}\neq m_{x,y}$. By Proposition~\ref{prop:qalphafreestrexp}, $$\Mol{M}= \preext{B_{\mathcal F(M)}}\subset \overline{\Gamma}$$ and so \eqref{eq:qalpha} holds also for every $m_{u,v}\in \Mol{M}\setminus \{m_{x,y}\}$. Thus, by Lemma \ref{lemma:ineqmolec}, we have that
\[ 2\frac{d(x,u)}{d(x,y)}\geq \norm{m_{x,y}-m_{u,y}} \geq 1-\rho(m_{x,y})\]
whenever $m_{x,y}\in \Gamma$ and $u\in M\setminus\{x,y\}$. In particular, the open ball centred at $x$ of radius $\frac{1-\rho(m_{x,y})}{2}d(x,y)$ is a singleton whenever $m_{x,y}\in \Gamma$. This means that the set
\[ A = \{x\in M\colon m_{x,y}\in \Gamma \text{ for some } y\in M\setminus\{x\} \}\]
is made up of isolated points. In order to prove that $A$ is dense in $M$, consider the Lipschitz function $f(t) = d(t,A)-d(0,A)$ for every $t\in M$, which belongs to $\Lip(M,\R)$, and consider its canonical linear extension $\hat{f}$ from $\mathcal F(M)$ to $\R$. Then, $\hat{f}$ vanishes on the norming set $\Gamma$, so $\hat{f}=0$. Thus, $f=0$, which yields that $\overline{A}=M$.
\end{proof}

\subsection{Relationship between the properties for Lipschitz-free spaces}\label{subsec:relations}
In the context of Lipschitz-free spaces over complete metric spaces, Figure~\ref{figure:Lipfree} contains the implications between the properties of the previous subsections.

\begin{figure}[h]
\centering
\begin{tikzpicture}[->,>=stealth',shorten >=1pt,auto,node distance=1.5cm,main node/.style={rectangle,
draw=black,
text width=6em,
            minimum height=2em,
text centered}]

\node[main node] (alpha) {Property $\alpha$};
\node[main node] (qalpha) at (4.65,-1.2) {Property quasi-$\alpha$};
\node[main node, text width=8em] (use) at (0,-6) {$\begin{array}{c} B_{\mathcal F(M)}=\cconv(S) \\  S \text{ unif.~str.~exp.}\end{array}$};
\node[main node, double, thick, text width=8em] (sna) at (4.65,-3.5) {$\SA(M, Y) $ dense for all $Y$};
\node[main node] (a) at (4.65, -5.2) {Property A};
\node[main node] (finite) at (9.3,0) {
$\begin{array}{c} \text{finite} \\  \text{dimensional}\end{array}$
};
\node[main node] (reflexive) at (9.3, -2.7) {Reflexive};
\node[main node] (rnp) at (9.3, -6) {RNP};

% implicaciones
\draw[-implies,double equal sign distance] (alpha) -- (qalpha) node[midway,swap,sloped]{(9)};
\draw[-implies,double equal sign distance] (alpha) -- (use) node[midway]{(5)};
\draw[implies-implies,double equal sign distance] (finite) -- (reflexive) node [midway]{(1)};
\draw[-implies,double equal sign distance,transform canvas={xshift=-0.4em}] (reflexive) -- (rnp);
\draw[-implies,double equal sign distance] (rnp) -- (sna) node[midway, sloped]{(4)};
\draw[-implies,double equal sign distance] (use) -- (sna) node[midway, sloped]{(6)};
\draw[-implies,double equal sign distance] (sna) -- (a) node[midway]{(8)};
\draw[-implies,double equal sign distance, transform canvas={yshift=-0.4em}] (finite) -- (alpha) node[midway]{(12)};
\draw[-implies,double equal sign distance] (qalpha) -- (sna) node[midway]{(10)};

% contraejemplos
\draw[-implies,double equal sign distance, neg] (use) -- (qalpha) node[pos=0.65, sloped]{(11)};
\draw[-implies,double equal sign distance, neg] (rnp) -- (qalpha) node[pos=0.65,sloped,swap]{(3)};
\draw[-implies,double equal sign distance, neg, transform canvas={xshift=0.4em}] (rnp) -- (reflexive) node[pos=0.4,swap]{(2)};
\draw[-implies,double equal sign distance, neg, transform canvas={yshift=0.4em}] (alpha) -- (finite) node[midway]{(13)};
\draw[-implies,double equal sign distance, neg] (rnp) -- (use) node[pos=0.55]{(7)};
\end{tikzpicture}
\caption{Relations between the sufficient conditions for property A in Lipschitz-free spaces}
     \label{figure:Lipfree}
\end{figure}
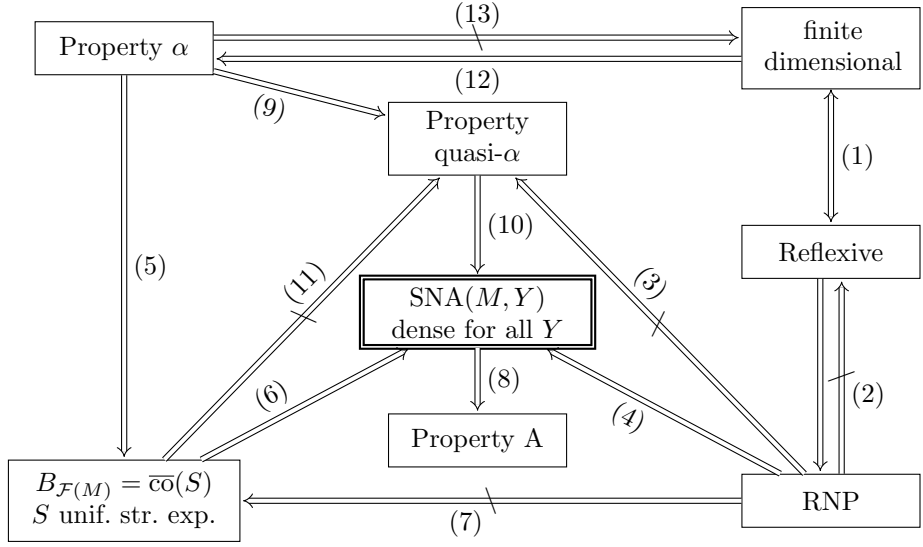

Let us discuss why the numbered implications and non-implications hold.

\noindent (1). It follows since every infinite-dimensional Lipschitz-free space contains an isomorphic copy of $\ell_1$ (this is folklore, but see \cite{cdw} where more is proved).

\noindent (2). $\mathcal F(\mathbb N)=\ell_1$.

\noindent (3). It follows from the following example.

\begin{example}\label{ex:RNPnotQalpha}
Let $M=([0,1],|\cdot|^\theta)$, where $0<\theta<1$. Then $\mathcal F(M)$ has the RNP (see Example \ref{ejernp}). Moreover, $M$ is concave  \cite[p.~51]{wea5}. By Proposition~\ref{qalpha+unifconc}, $\mathcal F(M)$ does not have property quasi-$\alpha$.
\end{example}

\noindent (4). It follows from Theorem \ref{teornpdensidad}, whose proof is based on the proof of Bourgain that asserts that RNP implies property A
\cite[Theorem~5]{bourgain1977}.

\noindent (5). It follows from the definition, introduced in \cite{Schachermayer}.

\noindent (6). It follows from Proposition \ref{prop:unifstrexpsnadens}, whose proof is based on \cite[Proposition~1]{lindens}, where it is proved that the existence of such a set $S$ implies property A.

\noindent (7). It follows from the following example.
\begin{example}\label{RNPnocufe} For every $n\in \N$, consider $M_n = \{0,x_n,y_n\}$, where $$d(0,x_n)=d(0,y_n)=1+1/n\qquad \text{and}\qquad d(x_n,y_n)=2$$ for each $n\in \mathbb N$, and let $M$ be its $\ell_1$-sum. Then, $\mathcal{F}(M)$ has the RNP, but $B_{\mathcal F(M)}$ is not the closed convex hull of any set of uniformly strongly exposed points.
\end{example}

\begin{proof}
First, $\mathcal{F}(M)$ has the RNP as it is the $\ell_1$-sum of finite-dimensional Banach spaces by Proposition \ref{prop:kaufmann}. Suppose that $B_{\mathcal F(M)} = \cconv(A)$. We claim that $m_{x_n,y_n}\in A\cup(-A)$ for every $n\in \mathbb N$. Indeed, assume that $m_{x_n,y_n}\notin A\cup(-A)$. Consider $f\colon M\longrightarrow \mathbb R$ given by $$f(0)=f(x_m)=f(y_m)=0\ \text{ if $m\neq n$},\ \ f(x_n)=-1,\ \text{ and } \ f(y_n)=1.$$ Clearly, $\norm{f}_L=1$. Moreover, we have that $$|\langle f, m_{x,y}\rangle | < (1+1/n)^{-1}\qquad \text{for every $m_{x,y}\in \Mol{M}\setminus\{m_{x_n,y_n}, m_{y_n,x_n}\}$.}$$ Thus, $A\cup(-A)$ is not norming, a contradiction.

Now, note that
\[ d(x_n,0)+d(y_n,0)-d(x_n,y_n) = \frac{2}{n}\]
goes to $0$ as $n$ goes to $\infty$, and so $A$ is not uniformly Gromov rotund.
\end{proof}

\noindent (8). It is obvious.

\noindent (9). It is obvious from the very definitions.

\noindent (10). It follows from Proposition \ref{prop:qalphaimplidensity}, whose proof is based on that of \cite[Proposition 2.10]{ChoiSong}, where it is proved that property quasi-$\alpha$ implies property A.

\noindent (11). It follows from the following example.

\begin{example}
		Let $M=([0,1],| \cdot |^\theta)$, where $0<\theta<1$. Then, $\Mol{M}$ is a set of uniformly strongly exposed points, $B_{\mathcal{F}(M)}=\cconv(\Mol{M})$, but $\mathcal F(M)$ fails to have property quasi-$\alpha$.
	\end{example}

	\begin{proof}
		In view of Proposition~\ref{qalpha+unifconc}, we just have to show that $\Mol{M}$ is uniformly Gromov rotund. To this end, define  the map $f \colon (0,1)\longrightarrow \mathbb{R}$ given by
		\[f(\lambda)=\frac{(1-\lambda)^\theta + \lambda^\theta -1}{\min\{(1-\lambda)^\theta, \lambda^\theta\}} \]
		It is easy to see that $0<\varepsilon:=\inf\{ f(\lambda) \colon \lambda \in (0,1)\}\leq 1$. Take different points $x,y,z \in [0,1]$ such that $x<y$. First, if we assume that $z < x$, then
		\[ \frac{(x,y)_z}{\min\{|x-z|^\theta, |y-z|^\theta\}} = \frac{|x-z|^\theta+|z-y|^\theta-|x-y|^\theta}{|x-z|^\theta}\geq \frac{|x-z|^\theta}{|x-z|^\theta}=1,\]
		and the same happens in the case of $y<z$. On the other hand, if $z=\lambda x + (1-\lambda)y$ for some $\lambda \in (0,1)$, then
\begin{align*}
\frac{(x,y)_z}{\min\{|x-z|^\theta, |y-z|^\theta\}} & = \frac{|x-y|^\theta(1-\lambda)^\theta + |x-y|^\theta \lambda^\theta -|x-y|^\theta}{|x-y|^\theta \min \{(1-\lambda)^\theta, \lambda^\theta\}} \\ &=\frac{(1-\lambda)^\theta+\lambda^\theta -1}{\min\{(1-\lambda)^\theta, \lambda^\theta\}}\geq \varepsilon.\qedhere
\end{align*}
\end{proof}

\noindent (12). See Example~\ref{ex:spacesAlpha}.

\noindent (13). $\mathcal F(\mathbb N)=\ell_1$ has property $\alpha$.

All the reverse implications not considered in our diagram (or which do not obviously follows from the ones given in it) are not known in the context of Lipschitz-free spaces. Particularly interesting are the cases of whether the converse of (4) and (8) holds.

\section{Weak density of $\SA(M,\R)$}\label{sectidensidadebil}

We have seen in the previous sections that the fact that $\SA(M,\R)$ is norm-dense in $\Lip(M,\R)$ imposes severe restrictions on the metric space $M$ (c.f.\ e.g.\ Theorem \ref{teo:length} or Corollary \ref{caraR-tree}), and even more the known sufficient conditions to get such density of Section \ref{sec:property_A}. However, that is not the case if we replace norm density with weak density, as the following theorem shows.

\begin{theorem}\label{th:weakdenseSA} Let $M$ be a complete pointed metric space. Then, $\SA(M,\R)$ is weakly sequentially dense in $\Lip(M,\R)$. Moreover, for every $g\in \Lip(M,\R)$ there is a sequence $\{g_n\}\subset\SA(M,\R)$ such that $g_n\stackrel{w}{\longrightarrow}g$, $\norm{g_n}_L\to\norm{g}_L$, and $g_n\to g$ uniformly on bounded sets.
\end{theorem}

This result extends \cite[Theorem 2.6]{kms}, where it was proved under the assumption of $M$ being a local metric space (equivalently, $M$ being a length space).

In order to prove our result we need a pair of lemmata. To begin with, the following result is implicitly proved in \cite[Theorem 2.6]{kms} under the assumption of $M$ being a length space, but thanks to Lemma \ref{lemmaweaknull}, we can show that the same argument works in a much more general setting.

\begin{lemma}\label{lemma:sufcond}
Let $M$ be a pointed metric space. Assume that there exists a sequence $\{B(x_n,r_n)\}_{n\in \N}$ of disjoint balls of $M$ and a sequence $\{y_n\}_{n\in \N}$ of points of $M$  such that $0< d(x_n,y_n)/r_n \to 0$ and $r_n\to 0$. Then for every $g\in \Lip(M,\R)$ there is a sequence $\{g_n\}\subset\SA(M,\R)$ such that $g_n\stackrel{w}{\longrightarrow}g$ and $\norm{g_n}_L\to\norm{g}_L$ and $g_n\to g$ uniformly.
\end{lemma}

\begin{proof} Given $g\in S_{\Lip(M,\R)}$, just follow the proof of Theorem 2.6 in \cite{kms} to construct a sequence $\{g_n\}$ in $\SA(M,\R)$ with $\supp(g_n-g)\subset B(x_n,r_n)$, $g_n(y_n)=g(y_n)$ and $\norm{g_n}_L =1+2\frac{d(x_n,y_n)}{r_n}\to 1$. Then, $\{g_n\}\stackrel{w}{\longrightarrow} g$ by Lemma \ref{lemmaweaknull}. Moreover,
	\begin{align*} |g_n(x)-g(x)| &\leq |g_n(x)-g_n(y_n)|+|g(y_n)-g(x)|
	\leq (\norm{g_n}_L+\norm{g}_L) d(y_n, x)\\
	&\leq (2+2\frac{d(x_n,y_n)}{r_n})(r_n+d(x_n,y_n))
	\end{align*}
whenever $x\in B(x_n,r_n)$. Since $r_n\to0$, $d(x_n,y_n)/r_n\to0$ and $\supp(g_n-g)\subset B(x_n,r_n)$, it follows that $g_n\to g$ uniformly. 		
\end{proof}

The following technical lemma will allow us to apply Lemma \ref{lemma:sufcond} in the case of $M$ being discrete but not uniformly discrete.

\begin{lemma}\label{lemadiscre}
Let $M$ be a complete metric space. Assume that $M$ is discrete but not uniformly discrete. Then, for every $k\geq 2$ and every $\varepsilon>0$, there exist $x, y\in M$ such that $0<d(x,y)\leq\varepsilon$ and the set $M\setminus B(x,k\,d(x,y))$ is not uniformly discrete.
\end{lemma}

\begin{proof}
Assume that there exist $k\geq 2$ and $\varepsilon>0$ such that
\[ \alpha(x,y) := \inf\{d(u,v)\colon u,v\in M\setminus B(x,kd(x,y)),\, u\neq v\}>0\]
whenever $0<d(x,y)\leq \varepsilon$. Since $M$ is not uniformly discrete, one can construct inductively two sequences $\{x_n\}$ and $\{y_n\}$ in $M$ such that $0<d(x_1,y_1)\leq \varepsilon$ and $0<d(x_{n+1},y_{n+1})\leq \min\{\alpha(x_n,y_n), 2^{-n-1}\varepsilon\}$ for every $n\in \mathbb N$. It follows that either $x_{n+1}\in B(x_n, kd(x_n, y_n))$ or $y_{n+1}\in B(x_n, kd(x_n,y_n))$. In any case,
	\[x_{n+1} \subset  B(x_n, kd(x_n,y_n)+d(x_{n+1},y_{n+1}))\subset B(x_n, 2^{-n}\varepsilon (k+1/2)).\]
 Thus, $\{x_n\}$ is Cauchy and so it has a limit in $M$, say $x$. Moreover, it is clear that $\{y_n\}$ also converges to $x$. Since $M$ is discrete, we conclude the existence of $n\in \mathbb N$ such that $x_n=y_n$, a contradiction.
\end{proof}

\begin{proof}[Proof of Theorem \ref{th:weakdenseSA}] We distinguish several cases depending on the properties of the set of cluster points $M'$. If $M'$ is infinite, then Lemma~\ref{lemma:sufcond} applies and so $\SA(M,\R)$ is weakly sequentially dense. Indeed, in such case it is not difficult to construct an infinite sequence of disjoint balls centered at (different) cluster points; as the centers are cluster points, we may also get the sequence $\{y_n\}_{n\in \N}$.

If $M'$ is empty, then we distinguish two more cases:
\begin{itemize}
\item If $M$ is uniformly discrete, then $\mathcal F(M)$ has the RNP (see  Example \ref{ejernp}), and so $\SA(M,\R)$ is indeed norm-dense in $\Lip(M,\R)$ by Theorem \ref{teornpdensidad}. Note that if $\norm{g_n-g}_L\to 0$ then $g_n\to g$ uniformly on bounded sets.
\item If $M$ is discrete but not uniformly discrete, then we can inductively apply Lemma~\ref{lemadiscre} to find sequences $\{x_n\}$, $\{y_n\}$ in $M$ such that, for every $n\in \mathbb N$, the space $M\setminus \bigcup_{m=1}^n B(x_m, 2md(x_m,y_m))$ is discrete but not uniformly discrete,  $x_{n+1},y_{n+1}\in M\setminus \bigcup_{m=1}^n B(x_m, 2md(x_m,y_m))$ and $$d(x_{n+1},y_{n+1})\leq\min\left\{\frac{n}{n+1}d(x_n, y_n), n^{-2}\right\}.$$ It is easy to check that the balls $\{B(x_n, nd(x_n,y_n))\}$ are pairwise disjoint and satisfy the requirement of Lemma~\ref{lemma:sufcond}. The conclusion follows.
\end{itemize}

It remains to consider the case when $M'$ is non-empty and finite, say $M'=\{a_1,\ldots, a_k\}$. Moreover, we may assume that $a_1=0$. Given $\varepsilon>0$, we denote $E_\varepsilon :=\bigcup_{i=1}^k \bigl[M\setminus B(a_i,\varepsilon)\bigr]$. If $E_\varepsilon$ is finite for every $\varepsilon>0$, then $M$ is compact and countable. Then, $\mathcal F(M)$ has the RNP (see Example \ref{ejernp}) and the conclusion follows. Thus, we may and do assume that there is $0<\varepsilon_0<\frac{1}{4}\min_{i\neq j}\{d(a_i,a_j)\}$ such that $E_{\varepsilon_0}$ is infinite. Moreover, note that $E_\varepsilon$ is discrete for every $\varepsilon>0$. If there is $0<\varepsilon\leq \varepsilon_0$ such that $E_\varepsilon$ is not uniformly discrete in $M$, then the same argument as above provides a sequence of disjoint balls such that Lemma~\ref{lemma:sufcond} applies. Thus, we may also suppose that $E_\varepsilon$ is infinite and uniformly discrete in $M$ for every $0<\varepsilon\leq \varepsilon_0$. By rescaling the metric space, we may assume that $\varepsilon_0=2^{-1}$. For $n\in \N$ and $i\in\{1,\ldots, k\}$, let us denote $C_n^i := E_{(n+1)^{-1}}\cap B(a_i, n^{-1})$ and
\[ \alpha_n^i:= \inf\{d(x,M\setminus\{x\})\colon x\in C_n^i\},\]
with the convention that $\inf\emptyset=+\infty$. Note, by passing, that $$M= E_{2^{-1}} \cup \bigcup_{n,k} C_n^i.$$
Now, we distinguish two cases.

\emph{Case 1}: assume that there is $i\in \{1,\ldots, k\}$ such that $\liminf_{n\to\infty} n\alpha_n^i=0$. Then we claim that it is possible to find a sequence $\{j_n\}$ in $\mathbb N$, and sequences $\{x_n\}$ and $\{y_n\}$ in $M$, such that:
\begin{enumerate}
\item $3nd(x_n, y_n)< (j_n+1)^{-1}$ for every $n$;
\item $4j_{n+1}^{-1} < (j_n+1)^{-1} -3nd(x_n,y_n)$ for every $n$;
\item $x_n\in C_{j_{n}}^i$.
\end{enumerate}
Indeed, take $j_1\geq 1$ such that
$6j_1\alpha_{j_1}^i< 1$. Then there is $x_1\in C_{j_1}^i$ such that $$3d(x_1, M\setminus \{x_1\}) < 2^{-1}j_1^{-1}\leq (j_1+1)^{-1}.$$ Thus, there is $y_1\in M$ with $3d(x_1,y_1)<(j_1+1)^{-1}$. Now, assume that we have defined $x_{n}$, $y_{n}$ and $j_n$, and let us define $x_{n+1}$, $y_{n+1}$ and $j_{n+1}$. By condition (1), we can take $j_{n+1}\in \mathbb N$ such that $$4j_{n+1}^{-1}<(j_n+1)^{-1} -3nd(x_n,y_n) \quad \text{ and } \quad 6nj_{n+1}\alpha_{j_{n+1}}^i < 1.$$ Then, there are $x_{n+1}\in C_{j_{n+1}}^i$ and $y_{n+1}\in M$ such that
\[ 3nd(x_{n+1}, y_{n+1})< 2^{-1}j_{n+1}^{-1}\leq (j_{n+1}+1)^{-1}.\]
This completes the construction of the sequences $\{x_n\}$, $\{y_n\}$ and $\{j_n\}$. Now, we claim that $$B(x_n, 3nd(x_n,y_n))\cap B(x_m, 3md(x_m, y_m)) = \emptyset$$ whenever $n\neq m$. Indeed, assume that $n<m$. It follows from (1) and (2) that
\[ B(x_n, 3nd(x_n,y_n))\cap B(a_i, j_{n+1}^{-1}) = \emptyset. \]
Moreover, from (1) and (3) it follows that
\begin{align*}
B(x_m, 3md(x_m, y_m)) &\subset B(x_m, 3(j_m+1)^{-1}) \\ &\subset B(a_i, 3(j_m+1)^{-1} + j_m^{-1}) \subset B(a_i, 4j_m^{-1}).
\end{align*}
Finally, note that $4j_{m}^{-1}\leq 4j_{n+1}^{-1} < (j_n+1)^{-1}$. Thus $B(x_m, 3md(x_m, y_m))$ is contained in $B(a_i, 4j_{n+1}^{-1})$ and so, it does not intersect $B(x_n, 3nd(x_n,y_n))$. Therefore, we can apply Lemma \ref{lemma:sufcond} to get that $\SA(M,\R)$ is weakly sequentially dense. This completes the proof in the first case.

\emph{Case 2}: assume now that there is a constant $C>0$ such that $C\leq n\alpha_n^i$ for every $n\in \mathbb N$ and $i\in\{1,\ldots, k\}$. We will show that in this case $\mathcal F(M)$ has the RNP. To this end, we will apply Proposition \ref{prop:kaufmann} several times in order to decompose $\mathcal F(M)$ as an $\ell_1$-sum of spaces with the RNP. Let us denote $E = E_{2^{-1}} \cup \{0\}$ and $N = \bigcup_{i=1}^k B(a_i, 1/2)$. We claim that $\mathcal F(M)$ is isomorphic to $\mathcal F(E) \oplus_1 \mathcal F(N)$.
Note that $N$ is bounded and so, $R=\sup\{d(x,0)\colon x\in N\}<+\infty$. Moreover, note also that $E$ is uniformly discrete in $M$ and so, $$\alpha:=\inf\{d(x,y)\colon x\in E, y\in N, x\neq y\}>0.$$ Thus, given $x\in E$ and $y\in N$, we have that
\[ d(x,0)+d(y,0) \leq d(x,y)+2d(y,0)\leq \left(1+2\frac{R}{\alpha}\right)d(x,y).\]
By applying Proposition \ref{prop:kaufmann}, we get the claim.

Now, for $i\in\{1,\ldots,k\}$, denote $\tilde{C}_0^i=\{0,a_i\}$, $\tilde{C}_n^i:= C_n^i\cup\{a_i\}$ if $n\geq 1$
and $\tilde{C}^i:=\bigcup_{n=0}^\infty \tilde{C}_n^i$. Note that $N= \bigcup_{i=1}^k \tilde{C}^i$ and $\tilde{C}^i\cap \tilde{C}^j=\{0\}$ if $i\neq j$. We claim that there is a constant $L>0$ such that
\[ d(x,0)+d(y,0)\leq L\, d(x,y) \]
whenever $x\in \tilde{C}^i$ and $y\in \tilde{C}^j$ with $i\neq j$. Take such an $x$ and $y$. Note that
\[\begin{split} d(x,y)&\geq d(a_i, a_j)-d(x,a_i)-d(y,a_j) \geq \frac{\min_{i\neq j}d(a_i,a_j)}{2}\\
& \geq \frac{\min_{i\neq j}d(a_i,a_j)}{4\diam(N)}(d(x,0)+d(y,0)).
\end{split}\]
Therefore, $L = \frac{4\diam(N)}{\min_{i\neq j}d(a_i,a_j)}$ does the work. This shows that
\[ \mathcal F(M) \approx \mathcal F(E)\oplus_1\mathcal F(N) \approx \mathcal F(E) \oplus_1 \mathcal F(\tilde{C}^1) \oplus_1 \cdots \oplus_1\mathcal F(\tilde{C}^k). \]
Finally, we will show that $\mathcal F(\tilde{C}^i)\approx \left[\bigoplus_{n=0}^\infty \mathcal F(\tilde{C}_{n}^i)\right]_{\ell_1}$ for every $i\in \{1,\ldots, k\}$. To this end, consider $a_i$ as the distinguished point in $\tilde{C}^i$ and notice that $\tilde{C}_n^i\cap \tilde{C}_m^i = \{a_i\}$ if $n\neq m$. Fix $n, m\in \mathbb N\cup\{0\}$ with $n<m$, take $x\in \tilde{C}_n^i$ and $y\in \tilde{C}_m^i$ with $x\neq y$. Then
\[
d(y,a_i) \leq \frac{1}{m} \leq \frac{\alpha_m^i}{C} \leq \frac{d(x,y)}{C}
\]
by definition of $\alpha_m^i$, and so
\[ d(x, a_i) + d(y,a_i) \leq d(x,y) + 2d(y,a_i) \leq (1+2C^{-1})d(x,y).\]
Thus, we can apply Proposition \ref{prop:kaufmann} to get that $\mathcal F(\tilde{C}^i)\approx \left[\bigoplus_{n=0}^\infty \mathcal F(\tilde{C}_{n}^i)\right]_{\ell_1}$. Therefore,
\[ \mathcal F(M) \approx \mathcal F(E) \oplus_1 \left[\bigoplus_{n,i} \mathcal F(\tilde{C}_{n}^i)\right]_{\ell_1} \]
where each one of the summands has the RNP as they are the Lipschitz-free space over a uniformly discrete metric space (see Example \ref{ejernp}).
\end{proof}

Let us finish the section with some observations.

\begin{remark}{\slshape
It follows from Theorem \ref{th:weakdenseSA} that the linear span of $\SA(M,\R)$ is norm dense in $\Lip(M,\R)$ for every metric space $M$. When $\mathcal{F}(M)$ has the RNP, one actually has that } $$\SA(M,\R)-\SA(M,\R)=\Lip(M,\R).$$
Indeed, it follows from \cite[Theorem 8]{bourgain1977} that in such a case $\SA(M,\R)$ contain a dense $G_\delta$-subset of $\Lip(M,\R)$, so $\SA(M,\R)$ is residual. It then follows that  $\SA(M,\R)-\SA(M,\R)=\Lip(M,\R)$ from Baire Category Theorem (indeed, an easy argument is given in \cite[Proposition 5.5]{KLMW2018}: just observe that for every $f\in \Lip(M,\R)$, $[f+\SA(M,\R)]\cap \SA(M,\R)$ is not empty since, otherwise, the second category set $f+\SA(M,\R)$ would be contained in the first category set $\Lip(M,\R)\setminus \SA(M,\R)$, which is impossible).
\end{remark}

Next, we observe that viewing $\Lip(M,\R)\equiv \mathcal{L}(\mathcal{F}(M),\R)\equiv \mathcal{F}(M)^*$, the Bishop-Phelps theorem gives that the set of those elements in $\Lip(M,\R)$ which attain their norm \emph{as elements of the dual of $\mathcal{F}(M)$} is always norm dense. On the other hand, $\SA(M,\R)$ is the set of elements in $\mathcal{F}(M)^*$ which attain their norm at a molecule. As the unit ball of $\mathcal{F}(M)$ is the closed convex hull of $\Mol{M}$, one may wonder whether Theorem \ref{th:weakdenseSA} actually follows from these facts, that is, if whenever a subset $A$ of a Banach space $X$ satisfies that $B_X=\overline{\co}(A)$, then the set of elements of $X^*$ which attain their norms at a point of $A$ is weakly dense on $X^*$. This is not true in general, as the following example shows.

\begin{example}
Let $X=c_0 \pten Y$ be the projective tensor product of $c_0$ and $Y$, where $Y$ is an equivalent renorming of $\ell_1$ such that $Y^*$ is strictly convex (see e.g.\ \cite[Theorem~II.2.6]{dgz}). We consider the subset of $B_X$ given by  $$A:=\{x\otimes y\colon  x\in B_{c_0},\, y\in B_{Y}\}$$ which satisfies that $B_X = \cconv(A)$ (see e.g.\ \cite[Proposition~2.2]{ryan}). Next, observe that if an element of $X^* \equiv \mathcal L(c_0, Y^{*})$ attains its norm at a point of $A$ then, in particular, it attains its norm as an operator from $c_0$ to $Y^*$ (actually more), that is, the set of elements of $X^*$ attaining their norms at points of $A$ is contained in $\NA(c_0,Y^*)$. But this set is not weakly dense since it is contained in the space of compact operators $\mathcal K(c_0, Y^{*})$ by \cite[Proposition 4]{lindens} and there are non-compact operators from $c_0$ to $Y^*$.
\end{example}

\section{Octahedrality of the bidual norm of Lipschitz-free spaces}\label{sectiocta}

As we have pointed out in the Introduction, it is proved in \cite[Theorem 2.4]{blr} that $\mathcal F(M)$ has an octahedral norm whenever $M$ is not uniformly discrete and bounded. The idea of the proof in \cite{blr} is to show that, under this hypothesis, every convex combination of weak-star slices of the unit ball of $\Lip(M,\R)$ has diameter two, and then use \cite[Theorem 2.1]{blroctajfa} to get the octahedrality of the predual $\mathcal{F}(M)$ of $\Lip(M,\R)$. The proof strongly relies on the fact that on bounded subsets of $\Lip(M,\R)$ there is a good characterisation of the weak-star convergence: it agrees with the pointwise convergence. If one wants to prove the octahedrality of the norm of the bidual space of $\mathcal{F}(M)$ for some $M$, the analogous way is to show that every convex combination of weak slices of the unit ball of $\Lip(M,\R)$ has diameter two and then use \cite[Corollary 2.2]{blroctajfa} to get from this fact the octahedrality of the norm of $\mathcal{F}(M)^{**}=\Lip(M,\R)^*$. The main difficulty for this is that, as far as we are concerned, the knowledge of the weak topology on bounded sets of $\Lip(M,\R)$ is not very satisfactory. Nevertheless, our Lemma \ref{lemmaweaknull} (jointly with \cite[Corollary 2.2]{blroctajfa}), allows us to provide the following result.

\begin{theorem}\label{teoctabiduallibre}
Let $M$ be a complete metric space. If $M'$ is infinite or $M$ is discrete but not uniformly discrete, then the norm of $\mathcal F(M)^{**}$ is octahedral.
\end{theorem}

\begin{remark}{\slshape
Note that the previous result is not sharp. For instance, it is well-known that the bidual norm of $\mathcal F(\mathbb N)=\ell_1$ is octahedral because every convex combination of slices of $B_{\ell_\infty}$ has diameter two \cite[Theorem 4.2]{aln}, but this result is not covered by the assumptions of our theorem.}
\end{remark}

As we announced above, in order to prove Theorem \ref{teoctabiduallibre}, we will focus on proving that $\Lip(M,\R)$ has the so-called SD2P. Recall that a Banach space has the \emph{SD2P} if every convex combination of weak slices of $B_X$ has diameter two. In fact, we will consider the following stronger notion, introduced in the recent paper \cite{anp}.

\begin{definition}
A Banach space $X$ has the \textit{symmetric strong diameter two property} (\emph{SSD2P} in short) if for every $n\in\mathbb N$, every slices $S_1,\ldots, S_n$ of $B_X$ and every $\varepsilon>0$, there are $x_i\in S_i$ for every $i\in\{1,\ldots, n\}$ and there exists $\varphi\in B_X$ with $\Vert \varphi\Vert>1-\varepsilon$ such that $x_i\pm \varphi\in S_i$ for every $i\in\{1,\ldots, n\}$.
\end{definition}

If $X$ is a dual Banach space, the weak-star version of the previous property (the \emph{weak-star-SSD2P}), defined in the natural way, was considered in \cite[Definition~5.1]{hlln}.

It is easy to prove (see e.g.\ \cite[Lemma 4.1]{aln}) that the SSD2P implies the SD2P, but the converse result is not true \cite[Remark 3.2]{hlln}.

Our next result is an abstract condition to get the SSD2P in certain spaces of Lipschitz functions.

\begin{lemma}\label{lematecnissd2p}
Let $M$ be a pointed metric space and assume that there exists a pair of sequences $\{x_n\},\{y_n\}$ in $M$ satisfying that $nd(x_n,y_n)\longrightarrow 0$ and that the balls $B(x_n, nd(x_n,y_n))$ are pairwise disjoint. Then, $\Lip(M,\R)$ has the SSD2P.
\end{lemma}

\begin{proof} By \cite[Theorem~2.1, (d)]{hlln}, it is enough to prove that given $N\in \N$ and norm-one Lipschitz functions $f_1,\ldots,f_N$, there are a weakly-null sequence $\{h_n\}_{n\in \N}$ of norm-one Lipschitz functions and sequences $\{g_i^n\}_{n\in\N}$ such that $\|g_i^n\|\longrightarrow 1$ and that $g_i^n\longrightarrow f_i$ weakly for every $i=1,\ldots,N$, satisfying that $\|g_i^n \pm h_n\|\longrightarrow 1$ for every $i=1,\ldots,N$.

Let us construct the sequences. Given $i\in\{1,\ldots, N\}$ and $n\in\mathbb N$, we define $$g_i^n\colon [M\setminus B(x_n,nd(x_n,y_n))]\cup B(x_n, n^\frac{2}{3}d(x_n,y_n))\longrightarrow \mathbb R$$ by the equation
$$g_i^n(x):=\left\{\begin{array}{cc}
f_i(x_n) & x\in B\left(x_n,n^\frac{2}{3}d(x_n,y_n)\right),\\
f_i(x) & x\notin B(x_n, nd(x_n,y_n)).
\end{array} \right.$$
It follows that $\Vert g_i^n\Vert_L\longrightarrow 1$. Indeed, pick $$x\in B(x_n,n^\frac{2}{3}d(x_n,y_n))\quad \text{and}\quad y\notin B(x_n, nd(x_n,y_n))$$ (the remaining cases are immediate). Then,
\begin{align*}
\frac{\vert g_i^n(x)-g_i^n(y)\vert}{d(x,y)}&=\frac{\vert f_i(x_n)-f_i(y)\vert}{d(x,y)}\leq \frac{d(x_n,y)}{d(x,y)}\leq \frac{d(x,y)+d(x,x_n)}{d(x,y)}\\
&=1+\frac{d(x,x_n)}{d(x,y)} \leq 1+\frac{n^\frac{2}{3}d(x_n,y_n)}{nd(x_n,y_n)}=1+\frac{1}{n^\frac{1}{3}}\longrightarrow 1.
\end{align*}
By McShane extension theorem, we can consider that the functions $g_i^n$ are defined in the whole of $M$. Now, observe that $\supp(g_i^n-f_i)\subseteq B(x_n,nd(x_n,y_n))$, which are disjoint sets by the assumption. Consequently, $g_i^n\stackrel{w}{\longrightarrow} f_i$ for every $i=1,\ldots,n$ by Lemma \ref{lemmaweaknull}.

Now, again the assumptions on the balls imply the existence of a sequence $\{h_n\}\subseteq \Lip(M,\R)$ such that, for every $n\in\mathbb N$, we have that $\Vert h_n\Vert_L=1$, that $\supp(h_n)\subseteq B(x_n,n^\frac{1}{3}d(x_n,y_n))$ and that $h_n(x_n)=0$. Again from the disjointness of the supports, we conclude that $\{h_n\}$ is weakly null from Lemma \ref{lemmaweaknull}. Let us prove that $\Vert g_i^n\pm h_n\Vert_L\longrightarrow 1$. To this end, pick $i\in\{1,\ldots, N\}$ and $n\in\mathbb N$. Given $x,y\in M$ with $x\neq y$, we have that
$$
\frac{\vert (g_i^n\pm h_n)(x)-(g_i^n\pm h_n)(y)\vert}{d(x,y)}\leq \underbrace{\frac{\vert g_i^n(x)-g_i^n(y)\vert}{d(x,y)}}_A+\underbrace{\frac{\vert h_n(x)-h_n(y)\vert}{d(x,y)}}_B=:C.
$$
Let us obtain an upper bound for $C$:
\begin{itemize}
\item If $B=0$ then $C\leq \Vert g_i^n\Vert_L$.
\item If $B\neq 0$, then either $x$ or $y$ belongs to $B(x_n,n^\frac{1}{3}d(x_n,y_n))$, so let us assume that such a point is $x$. In this case, notice that $g_i^n(x)=f(x_n)$. Now, if $y\in B(x_n,n^\frac{2}{3}d(x_n,y_n))$ then $g_i^n(y)=f(x_n)$ from where $A=0$ and $C\leq 1$ in this case. Otherwise, if $y\notin B(x_n,n^\frac{2}{3} d(x_n,y_n))$, then $h_n(y)=0$. Furthermore,
$$
d(x,y)\geq (n^\frac{2}{3}-n^\frac{1}{3})d(x_n,y_n).
$$
So, taking $n \geq 2$, we get
$$
B\leq \frac{d(x_n,x)}{d(x,y)}\leq \frac{n^\frac{1}{3}d(x_n,y_n)}{(n^\frac{2}{3}-n^\frac{1}{3})d(x_n,y_n)} =\frac{1}{n^\frac{1}{3}-1},
$$
and then $C\leq \Vert g_i^n\Vert_L+\frac{1}{n^\frac{1}{3}-1}$.
\end{itemize}
Taking supremum in $x,y\in M$ with $x\neq y$, we get that
$$
\Vert g_i^n\pm h_n\Vert_L\leq \Vert g_i^n\Vert_L +\frac{1}{n^\frac{1}{3}-1}
$$
for every $n \geq 2$. Since it is not difficult to prove that $\Vert g_i^n\pm h_n\Vert_L\geq 1$ (it is enough to consider points of $B(x_n, n^\frac{1}{3}d(x_n,y_n))$ from the assumption that $\Vert h_n\Vert_L=1$), we get that $\Vert g_i^n\pm h_n\Vert_L\longrightarrow 1$.
\end{proof}

As a consequence of the previous result we get the promised result about the SSD2P.

\begin{theorem}\label{teogenessd2pesc}
Let $M$ be a complete pointed metric space. If $M'$ is infinite or $M$ is discrete but not uniformly discrete, then $\Lip(M,\R)$ has the SSD2P.
\end{theorem}

Note that, as announced, Theorem \ref{teoctabiduallibre} follows from this result by \cite[Corollary 2.2]{blroctajfa}.

\begin{proof}[Proof of Theorem \ref{teogenessd2pesc}]
If $M'$ is infinite, it is not difficult to construct an infinite sequence of disjoint balls centered at (different) cluster points and use the fact that the center of the balls are cluster points to get a sequence $\{y_n\}_{n\in \N}$ that allows us to use Lemma~\ref{lematecnissd2p}.

On the other hand, if $M$ is discrete and not uniformly discrete, we can construct by Lemma~\ref{lemadiscre} a sequence of pairs $(x_n,y_n)$ such that, for every $n\in\mathbb N$, $0<d(x_n,y_n)<\frac{1}{n^2}$ and such that $B(x_n,nd(x_n,y_n))$ is a sequence of pairwise disjoint balls, so again Lemma~\ref{lematecnissd2p} applies.
\end{proof}

The following comment is pertinent.

\begin{remark}{\slshape
Note that Theorem \ref{teogenessd2pesc} improves, under its hypotheses, several known results about the big slice phenomena in spaces of Lipschitz functions. More precisely, given a metric space $M$ such that $M'$ is infinite or $M$ is discrete but not uniformly discrete, Theorem \ref{teogenessd2pesc} improves the consequences obtained in \cite{blr} (respectively, \cite{hlln}, \cite{ivak}) in $\Lip(M,\R)$, namely, that $\Lip(M,\R)$ has the weak-star-SD2P (respectively, weak-star-SSD2P, the property that every slice of its unit ball has diameter two).}
\end{remark}

In the compact case, we get the following optimal result.

\begin{corollary}
Let $M$ be an infinite compact metric space. Then the norm of $\mathcal F(M)^{**}$ is octahedral.
\end{corollary}

\begin{proof}
The case of $M'$ being infinite follows from Theorem \ref{teoctabiduallibre}. Otherwise, $M'$ is finite and then, $M$ is a countable compact metric space. Therefore, by \cite[Theorem~2.1]{dalet} there is a Banach space $Z$ such that $Z^{**}=\Lip(M,\R)$. Actually, $Z$ can be considered to be the so-called little Lipschitz space $\lip(M,\R)$, see \cite[Definition 3.1.1]{wea5} for background. Thus, $Z$ is a non-reflexive $M$-embedded Banach space by \cite[Theorem~6.6]{Kalton04}. As a consequence, both $Z$ and $Z^{**}=\Lip(M,\R)$ satisfy that every convex combination of weak slices of their unit ball has diameter two by \cite[Theorem~4.10]{aln}, and so the norm of $\mathcal F(M)^{**}=\Lip(M,\R)^*$ is octahedral by \cite[Corollary 2.2]{blroctajfa}.
\end{proof}

Let us discuss a little bit with a possible version of Theorem \ref{teogenessd2pesc} for vector-valued Lipschitz maps. We recall that given a metric space $M$ and a Banach space $Y$, it is said that the pair $(M,Y)$ satisfies the \emph{contraction-extension property} (\emph{CEP} in short) if McShane's extension theorem holds for $Y$-valued Lipschitz maps from subsets of $M$, that is, given $N\subseteq M$ and a Lipschitz function $f\colon N\longrightarrow Y$, there exists a Lipschitz function $F\colon M\longrightarrow Y$ which extends $f$ and satisfies that
$$\Vert F\Vert_{\Lip(M,Y)}=\Vert f\Vert_{\Lip(N,Y)}.$$
On the one hand, in the particular case of being $M$ a Banach space, the definition given above agrees with the one given in \cite{beli}. On the other hand, let us give some examples of pairs which have the CEP. First of all, given $M$ a metric space, the pair $(M,\mathbb R)$ has the CEP by McShane's extension theorem. Actually, the pair $(M,\ell_\infty(\Gamma))$ has the CEP for every set $\Gamma$. Another example coming from \cite[Chapter 2]{beli} is the fact that the pair $(H,H)$ has the CEP whenever $H$ is any Hilbert space. Anyway, the CEP is a restrictive property as, for instance, if $Y$ is a strictly convex Banach space such that there exists a Banach space $X$ with $\dim(X)\geq 2$ and verifying that the pair $(X,Y)$ has the CEP, then $Y$ is a Hilbert space \cite[Theorem~2.11]{beli}. See also \cite{Gode-NWEJM} for the relation between the extension of certain vector-valued Lipschitz maps and the approximation property of the Lipschitz-free spaces in the context of compact metric spaces.

Note that, following word-by-word the proof of Lemma \ref{lematecnissd2p}, using the CEP instead of McShane's extension theorem, we can get the following vector-valued version of Theorem \ref{teogenessd2pesc}.

\begin{theorem}\label{teogenessd2pvect}
Let $M$ be a pointed metric space and let $Y$ be a Banach space such that the pair $(M,Y)$ has the CEP. If $M'$ is infinite or $M$ is discrete but not uniformly discrete, then $\Lip(M,Y)$ has the SSD2P.
\end{theorem}

The previous theorem provides a partial positive answer to \cite[Question~3.1]{blr}, where the authors asked whether $\Lip(M,X^*)$ has the SD2P whenever the pair $(M,X^*)$ has the CEP and $M$ is not uniformly discrete and bounded.

\noindent \textbf{Acknowledgment:\ } The authors are very grateful to Vladimir Kadets, Colin Petitjean, and Anton\'{\i}n Proch\'{a}zka for many comments which have improved the final version of this paper. They also thanks the anonymous referee for the valuable suggestions which have also improved the exposition of the paper.

%\bibliographystyle{acm}
%\bibliography{biblio}

\begin{thebibliography}{10}

\bibitem{aln}
{\sc Abrahamsen, T., Lima, V., and Nygaard, O.}
\newblock Remarks on diameter 2 properties.
\newblock {\em J. Convex Anal.} 20 (2013), 439--452.

\bibitem{anp}
{\sc Abrahamsen, T., Nygaard, O., and Poldvere, M.}
\newblock New applications of extremely regular function spaces.
\newblock {\em Pacific J. Math.} (to appear), available at arXiv.org with reference
  \href{https://arxiv.org/abs/1711.01494}{1711.01494}.

\bibitem{AcostaRACSAM2006}
{\sc Acosta, M.~D.}
\newblock Denseness of norm attaining mappings.
\newblock {\em RACSAM} 100 (2006), 9--30.

\bibitem{ag}
{\sc Aliaga, R., and Guirao, A.}
\newblock On the preserved extremal structure of {L}ipschitz-free spaces.
\newblock {\em Studia Math.} 245 (2019), no. 1, 1--14.

\bibitem{AviMar}
{\sc Avil\'{e}s, A. and Mart\'{\i}nez-Cervantes, G.}
\newblock Complete metric spaces with property (Z) are length metric spaces.
\newblock {\em J. Math. Anal. Appl.} (to appear), doi:\href{https://doi.org/10.1016/j.jmaa.2018.12.051}{10.1016/j.jmaa.2018.12.051} .

\bibitem{blroctajfa}
{\sc Becerra~Guerrero, J., L\'opez-P\'erez, G., and Rueda~Zoca, A.}
\newblock Octahedral norms and convex combination of slices in {B}anach spaces.
\newblock {\em J. Funct. Anal.} 266 (2014), 2424--2435.

\bibitem{blr}
{\sc Becerra~Guerrero, J., L\'opez-P\'erez, G., and Rueda~Zoca, A.}
\newblock Octahedrality in {L}ipschitz free {B}anach spaces.
\newblock {\em Proc. Roy. Soc. Edinburgh Sect. A} 148 (2018), 447--460.

\bibitem{beli}
{\sc Benyamini, Y., and Lindenstrauss, J.}
\newblock {\em Geometric nonlinear functional analysis, vol.\ 1},
  {\em American Mathematical Society Colloquium Publications}, vol.~48.
\newblock American Mathematical Society, Providence, RI, 2000.

\bibitem{bourgain1977}
{\sc Bourgain, J.}
\newblock On dentability and the {B}ishop-{P}helps property.
\newblock {\em Israel J. Math.} 28 (1977), 265--271.

\bibitem{bh}
{\sc Bridson, M.~R., and Haefliger, A.}
\newblock {\em Metric spaces of non-positive curvature}, {\em
  Fundamental Principles of Mathematical Sciences} vol.~319.
\newblock Springer-Verlag, Berlin, 1999.

\bibitem{ChoiSong}
{\sc Choi, Y.~S., and Song, H.~G.}
\newblock Property (quasi-{$\alpha$}) and the denseness of norm attaining
  mappings.
\newblock {\em Math. Nachr.} 281 (2008), 1264--1272.

\bibitem{cdw}
{\sc C\'uth, M., Doucha, M., and Wojtaszczyk, P.}
\newblock On the structure of {L}ipschitz-free spaces.
\newblock {\em Proc. Amer. Math. Soc.} 144 (2016), 3833--3846.

\bibitem{dalet}
{\sc Dalet, A.}
\newblock Free spaces over countable compact metric spaces.
\newblock {\em Proc. Amer. Math. Soc.} 143 (2015), 3537--3546.

\bibitem{dkp}
{\sc Dalet, A., Kaufmann, P.~L., and Proch\'azka, A.}
\newblock Characterization of metric spaces whose free space is isometric
  $\ell_1$.
\newblock {\em Bull. Belg. Math. Soc. Simon Stevin} 23 (2016), 391--400.

\bibitem{deville}
{\sc Deville, R.}
\newblock A dual characterisation of the existence of small combinations of
  slices.
\newblock {\em Bull. Austral. Math. Soc.} 37 (1988), 113--120.

\bibitem{dgz}
{\sc Deville, R., Godefroy, G., and Zizler, V.}
\newblock {\em Smoothness and renormings in {B}anach spaces}, {\em
  Pitman Monographs and Surveys in Pure and Applied Mathematics}, vol.~64.
\newblock Longman Scientific \& Technical, Harlow, 1993.

\bibitem{Finet}
{\sc Finet, C.}
\newblock Uniform convexity properties of norms on a super-reflexive Banach space.
\newblock {\em Israel J. Math.} 53 (1986), 81--92.

\bibitem{LCthesis}
{\sc Garc\'ia~Lirola, L.}
\newblock {\em Convexity, optimization and geometry of the ball in Banach
  spaces}.
\newblock PhD thesis, Universidad de Murcia, 2017.
\newblock Available at \emph{DigitUM} with reference
  \href{http://hdl.handle.net/10201/56573}{http://hdl.handle.net/10201/56573}.

\bibitem{lppr}
{\sc Garc\'ia-Lirola, L., Petitjean, C., Proch\'azka, A., and Rueda Zoca, A.}
\newblock Extremal structure and duality of {L}ipschitz free spaces.
\newblock {\em Mediterr. J. Math.} 15 (2018), Art.~69, 23 pp.

\bibitem{gprstu}
{\sc Garc\'{\i}a-Lirola, L., Petitjean, C., and Rueda~Zoca, A.}
\newblock On the structure of spaces of vector-valued {L}ipschitz functions.
\newblock {\em Studia Math.} 239 (2017), 249--271.

\bibitem{gpr}
{\sc Garc\'ia-Lirola, L., Proch\'azka, A., and Rueda~Zoca, A.}
\newblock A characterisation of the {D}augavet property in spaces of
  {L}ipschitz functions.
\newblock {\em J. Math. Anal. Appl.} 464 (2018), 473--492.

\bibitem{Godard}
{\sc Godard, A.}
\newblock Tree metrics and their {L}ipschitz-free spaces.
\newblock {\em Proc. Amer. Math. Soc.} 138 (2010), 4311--4320.

\bibitem{godefroyocta}
{\sc Godefroy, G.}
\newblock Metric characterization of first {B}aire class linear forms and
  octahedral norms.
\newblock {\em Studia Math.} 95 (1989), 1--15.

\bibitem{Godefroy-survey-2015}
{\sc Godefroy, G.}
\newblock A survey on {L}ipschitz-free {B}anach spaces.
\newblock {\em Comment. Math.} 55 (2015), 89--118.

\bibitem{Gode-NWEJM} {\sc Godefroy, G.}
\newblock: Extensions of Lipschitz functions and Grothendieck's bounded approximation property.
\newblock {\em North-Western European J. of Math.} 1 (2015), 1--6.

\bibitem{Godefroy-PAFA-2016}
{\sc Godefroy, G.}
\newblock On norm attaining {L}ipschitz maps between {B}anach spaces.
\newblock {\em Pure Appl. Funct. Anal.} 1 (2016), 39--46.

\bibitem{gk}
{\sc Godefroy, G., and Kalton, N.~J.}
\newblock Lipschitz-free {B}anach spaces.
\newblock {\em Studia Math.} 159 (2003), 121--141.

\bibitem{hlln}
{\sc Haller, R., Langemets, J., Lima, V., and Nadel, R.}
\newblock Symmetric strong diameter two property.
\newblock Preprint (2018), available at arXiv.org with reference
  \href{https://arxiv.org/abs/1804.01705}{1804.01705}

\bibitem{ivak}
{\sc {Ivakhno}, Y.}
\newblock {Big slice property in the spaces of Lipschitz functions.}
\newblock {\em Visn. Khark. Univ., Ser. Mat. Prykl. Mat. Mekh.} 749
  (2006), 109--118.

\bibitem{KLMW2018}
{\sc Kadets, V., L\'opez, G., Mart\'{\i}n, M., and Werner, D.}
\newblock  Equivalent norms with an extremely nonlineable set of norm attaining functionals.
\newblock {\em J.~Inst.~Math.~Jussieu} (to appear), doi:\href{https://doi.org/10.1017/S1474748018000087}{10.1017/S1474748018000087}

\bibitem{ikw}
{\sc Ivakhno, Y., Kadets, V., and Werner, D.}
\newblock The {D}augavet property for spaces of {L}ipschitz functions.
\newblock {\em Math. Scand.\/} 101 (2007), 261--279. Corrigendum: (2009), {\em Math. Scand.\/} 104 (2009), 319.

\bibitem{kms}
{\sc Kadets, V., Mart\'{\i}n, M., and Soloviova, M.}
\newblock Norm-attaining {L}ipschitz functionals.
\newblock {\em Banach J. Math. Anal.} 10 (2016), 621--637.

\bibitem{Kalton04}
{\sc Kalton, N.~J.}
\newblock Spaces of {L}ipschitz and {H}\"older functions and their
  applications.
\newblock {\em Collect. Math.} 55 (2004), 171--217.

\bibitem{kaufmannPreprint}
{\sc Kaufmann, P.~L.}
\newblock Products of {L}ipschitz-free spaces and applications.
\newblock Old preprint version (2014), available at arXiv.org with reference
  \href{https://arxiv.org/abs/1403.6605}{1403.6605}

\bibitem{kaufmann}
{\sc Kaufmann, P.~L.}
\newblock Products of {L}ipschitz-free spaces and applications.
\newblock {\em Studia Math.} 226 (2015), 213--227.

\bibitem{lindens}
{\sc Lindenstrauss, J.}
\newblock On operators which attain their norm.
\newblock {\em Israel J. Math.} 1 (1963), 139--148.

\bibitem{Moreno}
\textsc{Moreno, J.~P.}
\newblock
Geometry of Banach spaces with $(\alpha,\varepsilon)$-property or
$(\beta,\varepsilon)$-property.
\newblock \emph{Rocky Mountain J. Math.} 27 (1997), 241--256.

\bibitem{pr}
{\sc Proch\'azka, A., and Rueda~Zoca, A.}
\newblock A characterisation of octahedrality in {L}ipschitz-free spaces.
\newblock {\em Ann. Inst. Fourier (Grenoble)} 68 (2018), 569--588.

\bibitem{ryan}
{\sc Ryan, R.~A.}
\newblock {\em Introduction to tensor products of {B}anach spaces}.
\newblock Springer Monographs in Mathematics. Springer-Verlag London, Ltd.,
  London, 2002.

\bibitem{Schachermayer}
{\sc Schachermayer, W.}
\newblock Norm attaining operators and renormings of Banach spaces.
\newblock {\em Israel J. Math.\/} 44 (1983), 201--212.

\bibitem{wea5}
{\sc Weaver, N.}
\newblock {\em Lipschitz algebras}.
\newblock World Scientific Publishing Co., Inc., River Edge, NJ, 1999.

\bibitem{Weaver-2017}
{\sc Weaver, N.}
\newblock On the unique predual problem for {L}ipschitz spaces.
\newblock {\em Math. Proc. Cambridge Philos. Soc.} 165 (2018), 467--473.

\bibitem{Yagoub}
{\sc Yagoub-Zidi, Y.}
\newblock Some isometric properties of subspaces of function spaces.
\newblock {\em  Mediterranean J. Math.} 10, 4 (2013), 1905--1915.

\end{thebibliography}

\end{document}